\documentclass[lettersize,journal]{IEEEtran}

\usepackage{amsmath,amsfonts}
\usepackage{algorithmic}
\usepackage{algorithm}
\usepackage{array}
\usepackage[caption=false]{subfig}
\usepackage{textcomp}
\usepackage{graphicx}
\usepackage{stfloats}
\usepackage{url}
\usepackage{verbatim}
\usepackage{cite}
\usepackage{hyperref}

\usepackage{bm} % bold math

\usepackage{amsthm} % theorems, proofs etc.

\usepackage{mathtools}
\usepackage{amssymb}
\usepackage{commath}
\usepackage{xcolor} % Text color
\usepackage{balance}
\usepackage[font={small}]{caption}
\usepackage{siunitx}

\let\labelindent\relax
\usepackage{enumitem} % enumerate without indent

\newcommand{\R}{\mathbb{R}}

\newtheoremstyle{mystyle}%                % Name
  {}%                                     % Space above
  {}%                                     % Space below
  {}%                                     % Body font
  {}%                                     % Indent amount
  {\bfseries}%                            % Theorem head font
  {.}%                                    % Punctuation after theorem head
  { }%                                    % Space after theorem head, ' ', or \newline
  {\thmname{#1}\thmnumber{ #2}\thmnote{ (#3)}}%                                     % Theorem head spec (can be left empty, meaning `normal')

\theoremstyle{mystyle}

\newtheorem{proposition}{Proposition}
\newtheorem{theorem}{Theorem}
\newtheorem{lemma}{Lemma}
\newtheorem{remark}{Remark}
\newtheorem{definition}{Definition}
\newtheorem{assumption}{Assumption}

\newtheorem{example}{Example}

\DeclareMathOperator*{\argmin}{argmin} % no space, limits underneath in displays
\begin{document}
%\small
 
\title{Optimization-free Smooth Control Barrier Function for Polygonal Collision Avoidance}
\author{Shizhen Wu,~\IEEEmembership{Graduate Student Member,~IEEE,}  Yongchun Fang$^*$,~\IEEEmembership{Senior Member,~IEEE,} Ning Sun,~\IEEEmembership{Senior Member,~IEEE,} Biao Lu,~\IEEEmembership{Member,~IEEE,} Xiao Liang,~\IEEEmembership{Senior Member,~IEEE,} and Yiming Zhao
\thanks{
%Manuscript received XXXXX XX, XXXX; accepted XXXXX XX, XXXX.
$^*$ Corresponding author: Yongchun Fang.}
\thanks{This work was supported in part by National Natural Science Foundation of China under Grants 62233011, U22A2050, 62203235 and 62273187. 
} 
\thanks{The authors are with the Institute of Robotics and Automatic Information
  System, College of Artificial Intelligence, Nankai University, Tianjin 300353,
  China, and also with the Institute of Intelligence Technology and Robotic
  Systems, Shenzhen Research Institute of Nankai University, Shenzhen 518083,
  China
(e-mail: szwu@mail.nankai.edu.cn;   fangyc@nankai.edu.cn; sunn@nankai.edu.cn;
 lubiao@mail.nankai.edu.cn; liangx@nankai.edu.cn;  zhaoym@mail.nankai.edu.cn)}}

%  \markboth{Journal of \LaTeX\ Class Files,~Vol.~14, No.~8, August~2021}%
%  {Shell \MakeLowercase{\textit{et al.}}: A Sample Article Using IEEEtran.cls for IEEE Journals}

\maketitle
% \thispagestyle{empty}
% \pagestyle{empty}

%%%%%%%%%%%%%%%%%%%%%%%%%%%%%%%%%%%%%%%%%%%%%%%%%%%%%%%%%%%%%%%%%%%%%%%%%%%%%%%%
%%%%%%%%%%%%%%%%%%%%%%%%%%%%%%%%%%%%%%%%%%%%%%%%%%%%%%%%%%%%%%%%%%%%%%%%%%%%%%%%
\begin{abstract}
  Polygonal collision avoidance (PCA) is short for the problem of collision avoidance between two polygons (i.e., polytopes in planar) that own their dynamic equations. This problem suffers the inherent difficulty in dealing with non-smooth boundaries and recently optimization-defined metrics, such as signed distance field (SDF) and its variants, have been proposed as control barrier functions (CBFs) to tackle PCA problems. In contrast, we propose an optimization-free smooth CBF method in this paper, which is computationally efficient and proved to be nonconservative. It is achieved by three main steps: a lower bound of SDF is expressed as a nested Boolean logic composition first, then its smooth approximation is established by applying the latest log-sum-exp method, after which a specified CBF-based safety filter is proposed to address this class of problems. To illustrate its wide applications, the optimization-free smooth CBF method is extended to solve distributed collision avoidance of two underactuated nonholonomic vehicles and  drive an underactuated container crane to avoid a moving obstacle respectively,  for which numerical simulations are also performed.

  \end{abstract}
  

    \section{Introduction}
    \label{sec:introduction}
 
    The control barrier function-based quadratic programming (CBF-QP) control method is popular for safe robotic control \cite{ames2016control,ames2019control, zhou2025temporal,dong2025security, xu2018constrained, tan2021high,xiong2022discrete}. 
    The CBF-based control can provide a simple and computationally
    efficient way for safe control synthesis \cite{ames2016control,ames2019control, zhou2025temporal,dong2025security},
    and it has been gradually extended to higher-order systems \cite{xu2018constrained, tan2021high,xiong2022discrete}. 
    Collision avoidance, i.e., driving the robot away from the obstacle and keeping a distance, is a
    common goal in reactive control of multi-agent robots such as  \cite{luo2021grpavoid,fu2020distributed}. And recently, CBFs have been gradually used to achieve more complex collision avoidance  \cite{funada2024collision,thirugnanam2023nonsmooth, dai2023safe, thirugnanam2022duality, singletary2022safety, wei2024diffocclusion}.
    When the shapes of robots and obstacles are complicated (instead of points or spheres), collision detection is not an obvious problem. 
    Since polytopes can non-conservatively approximate any convex shapes,
    compared with collision avoidance between ellipsoids or generic convex sets \cite{funada2024collision,thirugnanam2023nonsmooth, dai2023safe},  
    references \cite{thirugnanam2022duality, singletary2022safety, wei2024diffocclusion}
     are particularly interested in developing CBFs for collision avoidance between polytopes/polygons, where polygon refers particularly to a
     polytope in planar. Recently, related works about developing CBF methods for avoiding obstacles with irregular shapes can also be found in \cite{wu2024multiple}.

    More specifically, in \cite{thirugnanam2022duality, singletary2022safety,  
    wei2024diffocclusion}, researchers have noticed the idea of building CBFs via optimization-defined metrics for polytopic/polygonal collision avoidance. However, choosing such optimization-defined metrics as CBFs directly 
    meets the common difficulties: the resulting CBFs are implicit and not globally continuously differentiable in general.  
    To address these issues,  in \cite{thirugnanam2022duality}, minimum distance function/field (MDF) \cite{gilbert1985distance} is used to formulate a nonsmooth CBF, and a duality-based convex optimization approach is used to calculate the gradient of MDFs almost everywhere.  
    In \cite{singletary2022safety}, it has been pointed out that compared with MDFs, it is advantageous to define a CBF that is negative in the event of a collision, for which the signed distance field (SDF) \cite{cameron1986determining} is an ideal choice but implicit and nonsmooth. 
    Then the local linear approximation of the SDF is further chosen as a CBF in \cite{singletary2022safety}, avoiding the mentioned common difficulties but resulting in a conservative safe set. In \cite{wei2024diffocclusion}, a systematic method is proposed for circumscribing polygons by strictly convex shapes with tunable accuracy. Then a conic program-defined scaling factor is built to replace SDF to solve polygonal collision avoidance with application in occlusion-free visual servoing. 
    As a brief summary,  the exact solution of some convex optimization is still required when calculating the gradient of the proposed CBFs in \cite{thirugnanam2022duality, singletary2022safety, wei2024diffocclusion}, {which means that such kinds of methods depend on solving optimization online and the resulting CBF is still implicit.}  
    Hence, a natural issue is how to build a CBF for polytopic obstacle avoidance 
    that satisfies:  
    1) it is also a non-conservative smooth approximation of SDF; 
    2) it is optimization-free, i.e.,  there are no optimization algorithms embedded in calculating its gradient, so as to be more computationally efficient and construct CBF explicitly. Such two considerations motivate this paper.
     
    One potential approach to achieve the aforementioned two
    purposes is drawing inspiration from the latest work    \cite{han2023efficient}, but additional difficulties come after. 
    In \cite{han2023efficient}, for the purposes of gradient-based trajectory planning, 
    a lower boundary of SDF composed of $\max$/$\min$ operators is built first, after which an approximated smooth metric (called approx-SDF) is established by replacing $\max$/$\min$ with softmax/softmin operators respectively. 
    Such a two-step approximation method of SDF is computationally efficient since there are no convex optimizations embedded.
    It is natural to be aware that the approx-SDF has the potential to be a CBF satisfying the above two requirements. 
    However, the error between the approx-SDF and the original SDF is not discussed, in other words, the conservatism of choosing this approx-SDF as a CBF is still unknown.
    
    This paper goes on the topic of CBFs for real-time polygonal obstacle avoidance \cite{thirugnanam2022duality, singletary2022safety, wei2024diffocclusion}, and inspired by but different from the approx-SDF \cite{han2023efficient}, 
    we propose an optimization-free smooth CBF, where its non-conservatism is carefully studied and guaranteed.
    A brief comparison with existing methods is concluded in Tab. \ref{tab:comparisons}. 
    In detail, inspired by \cite{han2023efficient}, a lower bound of SDF between polygons is proposed and expressed as a nested Boolean logic composition at first, then the smooth approximation of the lower bound is established by applying the recently proposed softening method for any nested Boolean composition in \cite{molnar2023composing}, where the approximate error bound can be tuned by user-defined parameters. 
    The main contributions of this paper can be summarized as follows:
    
    1) Two important but not obvious facts are proved: 
      The proposed lower bound can represent the safety set non-conservatively, i.e., it has the same zero-level/-superlevel sets as SDF in Theorem \ref{the:the-signed-distance-is}, and it is  a nonsmooth barrier function candidate in the sense of  \cite{glotfelter2017nonsmooth}, as proved in Theorem \ref{lem:the-function-sd-candidate-NCBF}.
     
    2) It is proved that the proposed smooth approximated SDF could be chosen as a continuously differentiable CBF for
     polygons with single-integrator dynamics, based on which a specified CBF-QP safety filter is proposed to solve polygonal collision avoidance with provable safety, as concluded in Theorem \ref{the:conclusion}.
     
    3) Although the proposed method is formally stated for single-integrator dynamics, two more challenging  numerical simulation examples are given to show that it can be extended to underactuated
     nonholonomic vehicles \cite{dixon2001nonlinear} even more complex underactuated Euler-Lagrangian systems \cite{spong1998underactuated, yang2021adaptive} such as a container crane. 
     
    The paper is organized as follows. The CBF mwthod is briefly reviewed in Section \ref{sec:background}, and 
    problem formulation is given in Section \ref{sec:problem}. 
    In Section \ref{sec:synthesizing}, synthesizing a Boolean nonsmooth barrier function candidate is shown. The induced continuously differentiable CBF and 
    safety filter is shown in Section \ref{sec:safetyfilter}. Numerical simulation examples are performed in Section \ref{sec:numerical-simu}. 
    \begin{table*}[]
    \begin{center}
      \caption{  {A brief comparison between the existing methods  for PCA and the proposed one. For simplicity, the words  differentiable and optimization are abbreviated as diff and opti respectively. Since the expression of the proposed CBF 
      $\hat{h}_{a} $  is explicit, the values of itself $\hat{h}_{a} $ and  its gradient $\nabla \hat{h}_{a} $ can be calculated analytically and more efficiently.}}
      {\begin{tabular}{|l|l|l|l|l|l|}
      \hline
    Methods &  Smoothness  &  Conservatism  & Explicitness    & Applicable scene \\ \hline
    MDF-based CBF \cite{thirugnanam2022duality}      & Nonsmooth 
    & Nonconservative  &  Implicit (i.e., opti-embedded)   & Polytopes       \\ \hline 
    SDF-linearized CBF \cite{singletary2022safety}  & Smooth  & Conservative  &  Implicit  (i.e., opti-embedded)  &   Polytopes       \\ \hline
      Diff-opti CBF \cite{wei2024diffocclusion}  & Smooth  & Parameter-adjustable    &    Implicit   (i.e., opti-embedded)  & Polygons     \\ \hline
      Opti-free CBF $\hat{h}_{a} $ (proposed)    & Smooth   & Parameter-adjustable &  Explicit (i.e., opti-free)   &  Polygons           \\ \hline
      \end{tabular}
      \label{tab:comparisons}}
      \vspace{-0.6cm}
    \end{center}
    \end{table*}
    
    \section{Background}\label{sec:background}
    \subsection{{Control Barrier Function}}
     
    Consider affine control systems with state ${x \in \R^{n'}}$ and input ${u \in \R^{m'}}$:
    \begin{equation}
        \dot{x} = f(x) + g(x) u,
      \label{eq:system}
    \end{equation}
    where ${f: \R^{n'} \to \R^{n'}}$ and ${g: \R^{n'} \to \R^{n' \times m'}}$ are locally Lipschitz continuous.
    Given a control law ${k}({x})$,  
    the resulting closed-loop system is 
     \begin{align}
      \Sigma: \dot{{x}}={f}({x})+{g}({x}){k}({x}).
     \label{eq:closedloop}
     \end{align}
    For any initial condition ${x(0) = x_0}$, one assumes that the system~$\Sigma$ has a unique solution ${x(t)}$ existing for all ${t \geq 0}$.
  
    \begin{definition}\label{def:NBFcandidate}
    Given a nonempty closed set $\mathcal{C} \subset  \mathbb{R}^{n'}$ with no isolated point, 
    a  locally Lipschitz continuous  function $h: \mathbb{R}^{n'} \rightarrow \mathbb{R}$ is called a  nonsmooth barrier function (NBF) candidate on set $\mathcal{C}$,  if 
    $h(x) =0,~ \forall x\in \partial \mathcal{C}$; $h(x) >0,~ \forall x\in \operatorname{Int}(\mathcal{C})$; and $h(x) <0,~ \forall x\in \mathcal{C}^{c}:=\mathbb{R}^{n'}/\mathcal{C}$. 
    \end{definition}
   
    Definition \ref{def:NBFcandidate} is a tiny modified version of \cite[Definition 3]{glotfelter2017nonsmooth}. The condition $h(x) <0,~ \forall x\in \mathcal{C}^{c}$ is added to highlight that 
    $\mathcal{C}^{c}=\left\{x:  h(x) < 0\right\}\neq \emptyset$, which is essential to guarantee the stability of the set $\mathcal{C}$ \cite{ames2016control}.  To provide controllers with formal safety guarantees, the control barrier function method \cite{ames2016control} is introduced.
    
    \begin{definition}\label{def:cbf} 
    A continuously differentiable BF candidate $h: \mathbb{R}^{n'} \rightarrow \mathbb{R}$ on the  closed set $\mathcal{C} \subset  \mathbb{R}^{n'}$ is a valid control barrier function (CBF) for system \eqref{eq:system} on $\mathcal{D}$,  if there exists 
     an extended class-$\mathcal{K}$ function $\alpha: \mathbb{R} \rightarrow \mathbb{R}$, 
      an open set $\mathcal{D}$ that satisfies  $\mathcal{C} \subset \mathcal{D} \subseteq \mathbb{R}^{n'}$, such that for all $x \in \mathcal{D}$, the set $K(x)$ is non-empty, i.e., 
    \begin{equation}
     K(x):=\{ u\in \R^{m'}: \dot{h}(x,u) \geq-\alpha(h({x}))\} \neq \emptyset, \label{eq:minmathcalLF}
    \end{equation}
    where  $\dot{h}(x,u) = \nabla h(x)^{\top} (f(x) + g(x) u)$. 
    \end{definition}

    \begin{lemma} \label{lem:set-invariance-and-stability}
    If $h$ is a continuously differentiable CBF for system \eqref{eq:system} on $\mathcal{D}$, 
    then any locally Lipschitz continuous controller $k(\cdot)$ satisfying $k(x)\in K(x), \forall x \in \mathcal{D}$ renders: 
    the closed set $\mathcal{C}$ is forward invariant for \eqref{eq:closedloop}, i.e., $x_0 \in \mathcal{C}$ implies that $x(t) \in \mathcal{C},~\forall t \geq 0$ for every solution. 
    \end{lemma}
    
    The inequality in \eqref{eq:minmathcalLF} is often used as a constraint in optimization to synthesize safe controllers.
    For example, any Lipschitz continuous controller ${u_{0}(\cdot): \mathbb{R}^{n'} \to \mathbb{R}^{m'} }$ can be modified to a safe controller via the {\em quadratic program (QP)}:
    \begin{align}
    \begin{split}
        k(x) = \underset{u \in \mathbb{R}^{m'}}{\operatorname{argmin}} & \quad \| u - u_{0}(x) \|^2, \\
        \text{s.t.} & \quad  \dot{h}(x,u) \geq-\alpha(h({x})). 
    \end{split}
    \label{eq:QP}
    \end{align}
    This CBF-QP control paradigm is also known as {\em safety filter}. 
    
    \subsection{{Signed Distance Field}}\label{sec:sdf}
    
    The concept of signed distance field (SDF) can be introduced for any pair of convex sets  $\mathcal{P}^i, \mathcal{P}^j$. In detail, 
    the SDF between $\mathcal{P}^i, \mathcal{P}^j$ is defined as \cite{singletary2022safety,dai2023safe, cameron1986determining,  han2023efficient}:   
    \begin{equation}\label{eq:sd}
    \begin{aligned}
     \operatorname{sd}(\mathcal{P}^i, \mathcal{P}^j) &:= \operatorname{dist}(\mathcal{P}^i, \mathcal{P}^j) -  \operatorname{pen}(\mathcal{P}^i,\mathcal{P}^j),  \\
      \operatorname{dist}(\mathcal{P}^i, \mathcal{P}^j) 
    &:=\inf_{{t}}\{\|{t}\|:(\mathcal{P}^i+{t}) \cap  \mathcal{P}^j \neq \emptyset\}, \\
    \operatorname{pen}(\mathcal{P}^i, \mathcal{P}^j)
    &:= \inf _{{t}}\{\|{t}\|:(\mathcal{P}^i+{t}) \cap \mathcal{P}^j=\emptyset\}.
    \end{aligned}
    \end{equation}   
    In the above, $ \operatorname{dist}(\cdot)$ gives a nonnegative distance. 
    As concluded in \cite{singletary2022safety}, 
    it is more advantageous to pick $ \operatorname{sd}(\cdot)$ as a CBF than  $ \operatorname{dist}(\cdot)$, 
    since $ \operatorname{sd}(\cdot)$ is negative in the event of a collision, which can force the robot to re-approach the collision-free region.  
    If someone introduces the Minkowski sum
    $\mathcal{P}^j + \mathcal{P}^i 
    =\{a^j + a^i: a^j \in \mathcal{P}^j, a^i \in \mathcal{P}^i \}$ and 
    defines the set $- \mathcal{P}^i:= \{- a^i: a^i \in \mathcal{P}^i \}$, then the Minkowski difference
     of $\mathcal{P}^i,\mathcal{P}^j$ is defined as follows:  
    \begin{equation}\label{eq:mathcaiPij}
    \begin{aligned}
    \mathcal{P} &: = \mathcal{P}^j + (- \mathcal{P}^i) =\left\{a^j-a^i: a^j \in \mathcal{P}^j, a^i \in \mathcal{P}^i \right\}. 
    \end{aligned}
    \end{equation}
    By introducing the distance from origin $0$ to any closed set ${\mathcal{Q}}$: $d({0}, {\mathcal{Q}}):=\inf_{q\in \mathcal{Q}} \|0-q\| $, 
    one has 
\begin{equation}\label{eq:min-dist-primal2}
\operatorname{sd}(\mathcal{P}^i, \mathcal{P}^j) =  \operatorname{sd}({0},  \mathcal{P} ) 
=  \left\{ \begin{array}{l}~~d({0}, \partial {\mathcal{P}}),   {0} \notin \mathcal{P}, \\-d({0}, \partial {\mathcal{P}}),   {0} \in \mathcal{P} .
    \end{array} \right. 
\end{equation}  
The graphic illustration of Eq. \eqref{eq:sd}-\eqref{eq:min-dist-primal2} is given in Fig. \ref{sdf_minkowski}.  For any two polytopes (convex sets) $\mathcal{P}^i, \mathcal{P}^j$, their difference  $\mathcal{P}$ is still a polytope (convex set) \cite{cameron1986determining,  han2023efficient}.

\begin{figure}[!htp]
  \centering
       \vspace{-0.3cm}
       \includegraphics[width=0.7\linewidth]{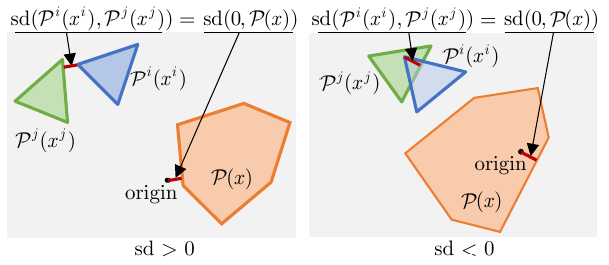}
       \caption{{The illustration of SDF between polygons (taking triangles for example), where $\mathcal{P}(x) 
    = \mathcal{P}^j(x^j) - \mathcal{P}^i(x^i)  $ denoting the Minkowski difference of $\mathcal{P}^i(x^i)$ and $\mathcal{P}^j(x^j)$.} }
       \label{sdf_minkowski}
         \vspace{-0.2cm}
\end{figure}
    
Another formula often used in the calculation and estimation of SDF is also introduced \cite{singletary2022safety,dai2023safe}:
$$\operatorname{sd}({0},  \mathcal{P} ) 
      = \max_{a^{\top} a=1}  \min _{\omega \in \mathcal{P}} ~ a^{\top}  \omega
      = \max_{a^{\top} a=1} - h_{\mathcal{P}}(a), $$
where $h_{\mathcal{P}}(a)=\max_{\omega \in \mathcal{P}} ~ a^{\top}  \omega$ represents the support function of the set $\mathcal{P}$ with outer normal vector $a$.

Specially, the SDF from a point $v$ to a halfspace $\mathbb{H}$ is 
\begin{align} 
  \operatorname{sd}(v,\mathbb{H})=\operatorname{sd}({0},v-\mathbb{H})  = a^{\top}v - b,  \label{eq:opersdvH} 
\end{align}
where $\mathbb{H}:=\left\{z: a^{\top} z \leq b \right\}$ with $\|a\|=1$ and  $b\in \mathbb{R}$. 
      
\section{Problem Formulation}\label{sec:problem}
    
Polygonal collision avoidance is short for the problem of collision avoidance between two polygons that own their dynamic equations. 
For simplicity of discussions,  consider two robots with the same single-integrator dynamics:
\begin{equation} \label{eq:N-affine-systems}
    \dot{x}^\mathfrak{i}(t) = u^\mathfrak{i}(t), \quad \mathfrak{i} \in \{i,j\},
\end{equation} 
where $x^\mathfrak{i} \in \R^{n} $,  $u^\mathfrak{i} \in  \R^{n}$ are the state and input associated with the  $\mathfrak{i}$-th robot ($n\geq2$). And the physical domain/shape associated to the $\mathfrak{i}$-th robot at $x^\mathfrak{i}$ is denoted as polygon $\mathcal{P}^\mathfrak{i}(x^\mathfrak{i})\subset \mathbb{R}^2$, as defined in Section \ref{sec:polygons}, after which the formal problem statement is shown in Section \ref{sec:problem-statement}.

\subsection{Expressions of Polygons}\label{sec:polygons}
    
Polygons, i.e., polytopes in planar, have a simpler geometric structure than generic polytopes.  From Euler's formula \cite[Theorem 2.3.10]{weibel2007minkowski}, polygons have as many vertices as edges. Let $r^{{\mathfrak{i}}}$ denote the number of vertices and edges of $\mathcal{P}^{\mathfrak{i}}(x^{\mathfrak{i}})$ for ${\mathfrak{i}}\in \{i,j\}$.  The halfspace($\mathcal{H}$)-representation is introduced to express $\mathcal{P}^{\mathfrak{i}}(x^{\mathfrak{i}})$: 
\begin{subequations} \label{eq:polytope-geometry}
\begin{align}
  \mathcal{P}^{\mathfrak{i}}(x^{\mathfrak{i}})  &= \left\{z: A^{\mathfrak{i}}(x^{\mathfrak{i}})z \leq b^{\mathfrak{i}}(x^{\mathfrak{i}}) \right\}  =  \bigcap_{k^{\mathfrak{i}}\in [r^{{\mathfrak{i}}}] }  \mathcal{P}^{{\mathfrak{i}}}_{k^{\mathfrak{i}}}(x^{\mathfrak{i}}) , \label{eq:polytope-geometry2} \\
  \mathcal{P}^{{\mathfrak{i}}}_{k^{\mathfrak{i}}}(x^{\mathfrak{i}}) &:= \left\{z \in \mathbb{R}^{2}:   (A^{{\mathfrak{i}}}_{k^{\mathfrak{i}}}(x^{\mathfrak{i}}))^{\top} z \leq b^{{\mathfrak{i}}}_{k^{\mathfrak{i}}}(x^{\mathfrak{i}})\right\}, \label{eq:polytope-half-space}
\end{align}
\end{subequations}
where  $A^{\mathfrak{i}}(x^{\mathfrak{i}}): \mathbb{R}^{n} \rightarrow \R^{r^{{\mathfrak{i}}} \times 2}$ and $b^{\mathfrak{i}}(x^{\mathfrak{i}}):  \mathbb{R}^{n}  \rightarrow \R^{r^{{\mathfrak{i}}}}$ represent the half spaces that define the geometry of robot ${\mathfrak{i}}$, and
    $(A^{{\mathfrak{i}}}_{k^{\mathfrak{i}}})^{\top},  b^{{\mathfrak{i}}}_{k^{\mathfrak{i}}}$ represent the $k^{\mathfrak{i}}$-th row of $ A^{{\mathfrak{i}}},  b^{{\mathfrak{i}}}$  respectively. 
For convenience, the equivalent\footnote{In this paper, both $\mathcal{H}$, $\mathcal{V}$-representations are considered to be known. The transformations from one to another have been well studied \cite{weibel2007minkowski}.} vertex($\mathcal{V}$)-representation of $\mathcal{P}^{\mathfrak{i}}(x^{\mathfrak{i}})$ is also introduced:
\begin{equation} \label{eq:polytope-geometry-2}
  \mathcal{P}^{\mathfrak{i}}(x^{\mathfrak{i}})  = \operatorname{co}\{\mathcal{V}^{\mathfrak{i}}(x^{\mathfrak{i}})\}, ~\mathcal{V}^{\mathfrak{i}}(x^{\mathfrak{i}}):=\{v^{{\mathfrak{i}}}_{k^{\mathfrak{i}}}(x^{\mathfrak{i}}): k^{\mathfrak{i}} \in  [r^{{\mathfrak{i}}}]\}, 
\end{equation}
where $\operatorname{co}\{\mathcal{V}^{\mathfrak{i}}\}$ denotes the convex hull of the vertex set $\mathcal{V}^{\mathfrak{i}}$, and $v^{{\mathfrak{i}}}_{k^{\mathfrak{i}}}(\cdot): \mathbb{R}^{n} \rightarrow \R^{2}$ denotes the $k^{\mathfrak{i}}$-th vertex of $\mathcal{P}^{\mathfrak{i}}(\cdot)$.  
The common assumptions are introduced for the following development. 
    
\begin{assumption}\label{eq:polytope-def}
For ${\mathfrak{i}}\in \{i,j\}$, the polygon $\mathcal{P}^{\mathfrak{i}}(x^{\mathfrak{i}})$ is always convex, bounded, and with a non-empty interior for all $x^{\mathfrak{i}}$. Additionally, it holds that:  \\
    1) Functions $A^{{\mathfrak{i}}}(\cdot),   b^{{\mathfrak{i}}}(\cdot)$  are continuously differentiable, and 
    the set of inequalities $A^{\mathfrak{i}}(x^{\mathfrak{i}})z \leq b^{\mathfrak{i}}(x^{\mathfrak{i}})$ does not contain any redundant inequality. \\
    2) Each column of $A^{\mathfrak{i}}(x^{\mathfrak{i}})$ belongs to the unit sphere $\mathbb{S}^{1}:=\{a\in\mathbb{R}^{2}:a^{\top}a=\|a\|^2=1     \}$. \\
    3) Functions $v^{{\mathfrak{i}}}_{k^{\mathfrak{i}}}, \forall k^{\mathfrak{i}} \in [r^{{\mathfrak{i}}}]$  are continuously differentiable, and 
    the Jacobian  $\partial v^{{\mathfrak{i}}}_{k^{\mathfrak{i}}}:=\frac{\partial v^{{\mathfrak{i}}}_{k^{\mathfrak{i}}}}{\partial x^{\mathfrak{i}}}$ is always row full rank.   
\end{assumption}
    
Conditions 2) and 3) are newly introduced for our development while others in Assumption \ref{eq:polytope-def} are derived from \cite{thirugnanam2022duality}. In detail,  condition 2) is used to simplify expressions without loss of generality;  
in condition 3),  \emph{row full rank} is introduced to guarantee the controllability of ${v}^{{\mathfrak{i}}}_{k^{\mathfrak{i}}}$ in   $\dot{v}^{{\mathfrak{i}}}_{k^{\mathfrak{i}}}= \partial v^{{\mathfrak{i}}}_{k^{\mathfrak{i}}}(x^{\mathfrak{i}})  \dot{x}^{\mathfrak{i}}$.  Next, examples are given to illustrate Assumption \ref{eq:polytope-def} more intuitive.   
     
\begin{example}[Polygonal rigid bodies]\label{exa:rigid_poly} 
Consider a polygonal rigid body $\mathcal{E}^\mathfrak{i} \subset \mathbb{R}^2$ centered at  point $p^\mathfrak{i} \in \mathbb{R}^2$ with attitude matrix $R(\theta^\mathfrak{i})=[\cos \theta^\mathfrak{i}, \sin \theta^\mathfrak{i}; - \sin \theta^\mathfrak{i}, \cos \theta^\mathfrak{i}]$ parameterized by angle $\theta^\mathfrak{i}$, i.e., $\mathcal{E}^\mathfrak{i}=p^\mathfrak{i} + R(\theta^\mathfrak{i}) \mathcal{G}^{\mathfrak{i}} $. 
Suppose $r^\mathfrak{i}$ is the number of vertexes of $ \mathcal{G}^{\mathfrak{i}}$ and the coordinate of the $k^\mathfrak{i}$-th vertex of $ \mathcal{G}^{\mathfrak{i}}$  in the body frame is  ${l}_{k^\mathfrak{i}}$ and ${l}_1, {l}_2, \ldots, {l}_{r^\mathfrak{i}}$ are   sorted clockwise with ${l}_{r^\mathfrak{i}+1}={l}_1$, then $\mathcal{V}$-representation is 
$$ \mathcal{E}^\mathfrak{i}({p}^\mathfrak{i},\theta^\mathfrak{i} ) =\left\{ {v}_{k^\mathfrak{i}}={p}^\mathfrak{i}+{R}(\theta^\mathfrak{i}) {l}_{k^\mathfrak{i}} \in \mathbb{R}^{2}, k^\mathfrak{i}=1,2, \ldots, r^\mathfrak{i} \right\},$$ 
and the $\mathcal{H}$-representation is determined as follows:
\begin{subequations}\label{mathbfHehe}
\begin{align}
      A^\mathfrak{i}_{k^\mathfrak{i}}(p^\mathfrak{i},\theta^\mathfrak{i})  & 
      :=\frac{{B}\left({v}_{k^\mathfrak{i}+1}-{v}_{k^\mathfrak{i}}\right)}{\left\|{v}_{k^\mathfrak{i}+1}-{v}_{k^\mathfrak{i}}  \right\|}
      = \underbrace{{B}{R}(\theta^\mathfrak{i})}_{{R}(\theta^\mathfrak{i}+\frac{\pi}{2})}   \frac{\Delta {l}_{k^\mathfrak{i}} }{\left\|\Delta {l}_{k^\mathfrak{i}} \right\|} 
      , \label{mathbfHe} \\  
      b^\mathfrak{i}_{k^\mathfrak{i}}(p^\mathfrak{i},\theta^\mathfrak{i}) & 
      := A_{k^\mathfrak{i}}^{\top}{v}_{k^\mathfrak{i}}  
      = A_{k^\mathfrak{i}}^{\top} \left( {p}^\mathfrak{i}+{R}({\theta}^\mathfrak{i}) {l}_{k^\mathfrak{i}}  \right), \label{mathbfhe}
\end{align}
\end{subequations}
where ${B}:=R(\frac{\pi}{2})=\left[0, -1; 1,0\right]$ and $\Delta {l}_{k^\mathfrak{i}}:= {l}_{k^\mathfrak{i}+1}- {l}_{k^\mathfrak{i}}$.  Denoting $x^\mathfrak{i}=(p^\mathfrak{i},\theta^\mathfrak{i})$, one obtains 
$$\begin{aligned}
  \frac{\partial A^\mathfrak{i}_{k^\mathfrak{i}}}{\partial x^\mathfrak{i}} & =\left[0,0, R({\pi}/{2}) A^\mathfrak{i}_{k^\mathfrak{i}} \right],  ~ \frac{\partial v^\mathfrak{i}_{k^\mathfrak{i}}}{\partial x^\mathfrak{i}}=\left[ I_2, R(\theta^\mathfrak{i} + {\pi}/{2}) l_{k^\mathfrak{i}} \right],  \\
 \frac{\partial b^\mathfrak{i}_{k^\mathfrak{i}}}{\partial x^\mathfrak{i}}  &= \left[ (A^\mathfrak{i}_{k^\mathfrak{i}})^{\top}, \left((v^\mathfrak{i}_{k^\mathfrak{i}})^{\top}R(\theta^\mathfrak{i}+\pi) + l_{k^\mathfrak{i}} ^{\top}\right)  \Delta l_{k^\mathfrak{i}}/\|\Delta l_{k^\mathfrak{i}}\| \right].   %_{1\times 3},                            
\end{aligned} $$
Based on the above expressions,  it is not difficult to verify the conditions  in  Assumption \ref{eq:polytope-def}. 
\end{example}
    
\begin{example}[Polygonal formation]\label{exa:formation} 
Consider three agents with position $v_\mathfrak{i}\in\mathbb{R}^2,\mathfrak{i}=1,2,3$. Let $x=[v_1^{\top},v_2^{\top},v_3^{\top}]^{\top}$, 
then the physical domain occupied by the triangular formation can be regarded as $\mathcal{E}(x)=\operatorname{co}\{v_1,v_2,v_3\}$.  Similar to 
Example \ref{exa:rigid_poly},  the $\mathcal{H}$-representation of $\mathcal{E}(x)$ can be obtained:   
\begin{align*}
      A_\mathfrak{i}(x)  & := {{B} \Delta v_\mathfrak{i}}/{\left\|\Delta v_\mathfrak{i} \right\|} , ~ b_{\mathfrak{i}}(x) 
      := A_{\mathfrak{i}}(x)^{\top}{v}_{\mathfrak{i}}, 
\end{align*}
for all $\mathfrak{i}=1,2,3$, where $\Delta v_\mathfrak{i}:={v}_{\mathfrak{i}+1}-{v}_{\mathfrak{i}} $ and $v_4=v_1$. If they are not located in the same line, $\mathcal{E}(x)$  satisfies all conditions in Assumption \ref{eq:polytope-def} obviously.
\end{example}
    
\subsection{Problem Statement}\label{sec:problem-statement}
    
For simplicity of notation, let $x := [(x^i)^{\top},(x^j) ^{\top}]^{\top}\in \mathbb{R}^{2n}$, $u := [(u^i)^{\top}, (u^j )^{\top}]^{\top}$, then  the evolution of  systems $i$ and $j$ in \eqref{eq:N-affine-systems} can be abbreviated as  
\begin{equation} \label{eq:dotxfxgxu}
\dot{x}= u.
\end{equation} 
Define a safe set $\mathcal{S}$ as the  zero-superlevel set of the signed distance between robots $i$ and $j$: 
\begin{equation} \label{eq:def-safe-set}
\mathcal{S}:= \{ x: h_s(x) \geq  0 \},
\end{equation}
where $ h_s(x):=\operatorname{sd}(\mathcal{P}^i(x^i), \mathcal{P}^j(x^j))$, as defined in \eqref{eq:sd}. The closed-loop system is considered safe w.r.t. $\mathcal{S}$ if the obstacle collision is avoided, i.e., 
\begin{equation} \label{eq:def-safe-set2}
    x(t)=  \left[(x^i(t))^{\top},(x^j(t))^{\top}\right]^{\top} \in \mathcal{S}, ~ \forall \; t\geq 0. 
\end{equation}

\textbf{Polygonal Collision Avoidance:} 
Consider the system \eqref{eq:N-affine-systems} whose shapes are polygons \eqref{eq:polytope-geometry} that satisfy Assumption \ref{eq:polytope-def}. The goal is to 
design a control law for $u$ such that the closed-loop system is safe w.r.t. $\mathcal{S}$ in the sense of \eqref{eq:def-safe-set2}. 

\begin{remark}
The studied problem is similar to the polytopic/polygonal collision avoidance in \cite{thirugnanam2022duality, singletary2022safety, wei2024diffocclusion}, and for brevity, it is assumed that both two subsystems have single-integrator dynamics. 
It will be shown in Section \ref{sec:numerical-simu} that the proposed solution could also be extended to nonlinear systems even two-order underactuated systems.  $\hfill\square$ 
\end{remark}

\section{Synthesizing a Boolean NBF Candidate}\label{sec:synthesizing} 
In this section, the detailed procedure of synthesizing a Boolean nonsmooth barrier function candidate is shown, which is composed of two main steps: establishing a lower bound of SDF and then building a boolean NBF candidate from such a lower bound. 
   
\subsection{Establishing a lower bound of SDF}
 We begin the development with \eqref{eq:min-dist-primal2}. First, associated with the support function $h_{\mathcal{P}}(a)$, the support halfspace of the set $\mathcal{P}$ with outer normal vector $a$ is introduced: 
\begin{equation} \label{eq:support-halfspace}
     H^{-}_{\mathcal{P}}(a) :=\left\{\omega \in \mathbb{R}^{2}: a^{\top} \omega \leq h_{\mathcal{P}}(a)\right\}, ~\forall a \in \mathbb{S}^{1}.
\end{equation} 
For such a set whose boundary is a hyperplane, it is easy to calculate its SDF at origin  via \eqref{eq:opersdvH}:   
\begin{equation*} 
\operatorname{sd}({0}, H_{\mathcal{P}}^{-}(a)) =a^{\top} 0 - h_{\mathcal{P}}(a) =- h_{\mathcal{P}}(a). 
\end{equation*} 
Based on the formula \eqref{eq:min-dist-primal2}, the SDF $h_{s}(x)$ in \eqref{eq:def-safe-set} can be calculated by a mini-maximum optimization: 
\begin{equation*} 
  h_{s}(x) =  \max_{\|a\|=1} \min _{\omega \in \mathcal{P}(x)}   a^{\top}  \omega
  = \max_{a^{\top}a=1}   \operatorname{sd}({0}, H_{\mathcal{P}(x)}^{-}(a)),  \label{eq；hsxmaxa1}
\end{equation*}
which is nonsmooth and hard to calculate exactly due to the $\max$/$\min$ operators on compact sets \cite{singletary2022safety,dai2023safe}. 
    
To address this, other facts about the difference $\mathcal{P}$ is illustrated first. Combing \eqref{eq:polytope-geometry2} with \eqref{eq:mathcaiPij}, it leads to 
\begin{subequations}  \label{eq:mathcalPij}
\begin{align} 
\mathcal{P} &=  \mathcal{P}^j - \bigcap_{k^i\in [r^{i}] }  \mathcal{P}^{i}_{k^i}
    \subseteq \bigcap_{k^i\in [r^{i}] } \left( \mathcal{P}^j - \mathcal{P}^{i}_{k^i} \right) ,\\
\mathcal{P} &=  \bigcap_{k^j\in [r^{j}] }    \mathcal{P}^{j}_{k^j} - \mathcal{P}^i 
    \subseteq  \bigcap_{k^j\in [r^{j}] }   \left( \mathcal{P}^{j}_{k^j} - \mathcal{P}^i \right) , 
\end{align}
\end{subequations}
where the set inclusion "$\subseteq$" comes from the fact \cite[Eq. (3.2)]{schneider2014convex}: $(A\cap B)-C \subseteq (A-C)\cap(B-C)$ holds  for arbitrary subsets $A,B,C$. To avoid the redundancy of symbols, sets $ \mathcal{P}^j - \mathcal{P}^{i}_{k^i} $ with  index ${k^i}$ and $\mathcal{P}^{j}_{k^j} - \mathcal{P}^i$ with index ${k^j}$, are together renumbered as sets  $\mathcal{Q}_{k}$ with a single index $k$:   
\begin{align*}
  \mathcal{Q}_{k}(x) :=
     \left\{\begin{array}{l}
    \mathcal{P}^j(x^j) - \mathcal{P}^{i}_{k^i}(x^i), k={k^i}, ~~~~~~
    \forall {k^i} \in [r^{i}], \\
    \mathcal{P}^{j}_{k^j}(x^j) - \mathcal{P}^i(x^i), k= r^{i} + {k^j}, 
      \forall {k^j} \in [r^{j}],
\end{array}\right. 
\end{align*}
 based on which \eqref{eq:mathcalPij} can be rewritten as 
    \begin{align} 
    \mathcal{P}(x) 
    &\subseteq   \bigcap_{k \in [r^{i}+ r^{j}]}  \mathcal{Q}_{k}(x).  
     \label{eq:mathcalPij2}
    \end{align}
Recalling the fact about SDF in \cite[Theorem 2.1]{delfour2011shapes} that: 
    $A\subseteq B \Rightarrow  \operatorname{sd} ({0},  A) \geq  \operatorname{sd} ({0},  B)$, the following inequality holds naturally. 
   
    \begin{proposition}\label{lem:hsxhax}
    The following defined function $h_{a}(x)$ is one lower boundary of $h_{s}(x)$: 
    \begin{equation} \label{eq:mathcalPij4}
    \begin{aligned} 
    h_{s}(x) &= \operatorname{sd} (0, \mathcal{P}(x))
     = \max_{a^{\top}a=1}   \operatorname{sd}({0}, H_{\mathcal{P}(x)}^{-}(a)) \\
    &\geq \max_{k \in [r^{i}+ r^{j}]} \operatorname{sd}(0, \mathcal{Q}_{k}(x))=: h_{a}(x).   
    \end{aligned}
    \end{equation}
    \end{proposition}
    
    Recently, the softening version of function $h_a(x)$ has been applied in \cite{han2023efficient} to substitute $h_s(x)$ to achieve efficient optimization-based trajectory planning. 
    However, the following problems still have not been studied, which are not obvious to answer:   
    \begin{itemize}
      \item [1)] Are the sets $\mathcal{Q}_{k}(x), \forall k \in [r^{i}+ r^{j}]$ some support halfspaces of $\mathcal{P}(x)$? Can the equality in  \eqref{eq:mathcalPij2} hold?
    
      \item [2)] Is $h_{a}(x)$ non-conservative? Namely,  are their zero-superlevel  sets $\mathcal{S}_{a}:=\{x: h_a(x)\geq0\} $ and $\mathcal{S}=\{x: h_s(x)\geq0\}$ the same?
    
      \item [3)] Is it hopeful to choose $ h_{a}(x)$ as a Barrier function candidate on safety set $\mathcal{S}$?   
    \end{itemize}
  The answers to the above three questions can be found in the following Lemma \ref{lem:support-halfspace},  Theorem \ref{the:the-signed-distance-is} and Theorem \ref{lem:the-function-sd-candidate-NCBF} respectively.

\subsection{Building a boolean NBF candidate from the lower bound}\label{sec:cbNBFc}
    
The next attempts to apply the properties of polytopes $\mathcal{P}^{i}, \mathcal{P}^{j}$, so to reformulate the function ${h}_{a}(x)$ as a  
Boolean composition expression.

From \eqref{eq:polytope-geometry} and definitions of Minkowski difference,  it holds  
    \begin{align}
    \mathcal{P}^j - \mathcal{P}^{i}_{k^i}
    &= \left\{z={\omega}^{j} - {\omega}^{i}: {\omega}^{j} \in \mathcal{P}^{j}, {\omega}^{i} \in \mathcal{P}^{i}_{k^i}\right\} \label{eq:matPjPiki} \\
    & =\left\{z: (- A^{i}_{k^i})^{\top} z \leq b^{i}_{k^i} + (-A^{i}_{k^i})^{\top} {\omega}^{j}, {\omega}^{j} \in \mathcal{P}^{j}\right\} \nonumber  \\
    & =\left\{z: (-A^{i}_{k^i})^{\top} z \leq b^{i}_{k^i} + \max_{{\omega}^{j} \in \mathcal{P}^{j}}  (-A^{i}_{k^i})^{\top} {\omega}^{j}  \right\} , \nonumber 
    \end{align}
     for all $k^i\in [r^{i}]$. Similarly,  for all $k^j\in [r^{j}]$, 
    \begin{align}
    \mathcal{P}^{j}_{k^j} - \mathcal{P}^i 
    &= \left\{z={\omega}^{j} - {\omega}^{i}: {\omega}^{j} \in \mathcal{P}^{j}_{k^j}, {\omega}^{i} \in \mathcal{P}^{i}\right\} 
    \label{eq:matPjkjPi} \\
    & =\left\{z: (A^{j}_{k^j})^{\top} z \leq b^{j}_{k^j} -  (A^{j}_{k^j})^{\top} {\omega}^{i}, {\omega}^{i} \in \mathcal{P}^{i}\right\} \nonumber  \\
    & =\left\{z: (A^{j}_{k^j})^{\top} z \leq b^{j}_{k^j} + \max_{{\omega}^{i} \in \mathcal{P}^{i}}   - (A^{j}_{k^j})^{\top} {\omega}^{i}\right\}.  \nonumber 
    \end{align}
    Based on the above expressions and recalling the notation of the support halfspace 
    $  H^{-}_{\mathcal{P}}(\cdot) $ in \eqref{eq:support-halfspace}, one can know from the following lemma that $ \mathcal{P}^j - \mathcal{P}^{i}_{k^i}$, $\mathcal{P}^{j}_{k^j} - \mathcal{P}^i$ are support halfspaces of $\mathcal{P}$ with outer normal vectors $ -A^{i}_{k^i}, A^{j}_{k^j}$ respectively. 
   
    \begin{lemma}\label{lem:support-halfspace} 
    For all $k^i\in [r^{i}]$ and $k^j\in [r^{j}]$, 
    one has 
    \begin{align}
     H_{\mathcal{P}}^{-}( -A^{i}_{k^i} )=\mathcal{P}^j - \mathcal{P}^{i}_{k^i}, ~~
     H_{\mathcal{P}}^{-}(  A^{j}_{k^j} )=\mathcal{P}^{j}_{k^j} - \mathcal{P}^i.  \label{eq:mathcalPkijwu}
    \end{align}
    Collect outer normal vectors of $-\mathcal{P}^i, \mathcal{P}^j$  respectively: %the finite sets:
    \begin{subequations}\label{eq:mthcalN-ij}
    \begin{align}
    \mathcal{N}_{-}^{i}(x^i)&:= \left\{-A^{i}_{k^i}(x^i): {k^i} \in [{r^i}]\right\}, \\
    \mathcal{N}^{j}(x^j) &:= \left\{A^{j}_{k^j}(x^j): {k^j} \in [{r^j}] \right\}, 
    \end{align}
    \end{subequations}
    then the equality in \eqref{eq:mathcalPij2} holds, i.e., 
    \begin{align} 
    \mathcal{P}(x) 
    =\bigcap_{a \in \mathcal{N}(x) }     H^{-}_{\mathcal{P}(x)}(a) 
    = \bigcap_{k \in [r^{i}+ r^{j}]}  \mathcal{Q}_{k}(x) ,  
     \label{eq:mathcalPij5} 
    \end{align} 
    where $\mathcal{N}(x):=\mathcal{N}_{-}^{i}(x^i) \cup  \mathcal{N}^{j}(x^j) $.
    \end{lemma}
    
    \begin{proof}
    The proof is reported in Appendix \ref{app:proof-lem2} to avoid breaking the flow of the exposition. 
    \end{proof}

    \begin{figure}[!htp]
       \centering
       \vspace{-0.4cm}
       \includegraphics[width=0.7\linewidth]{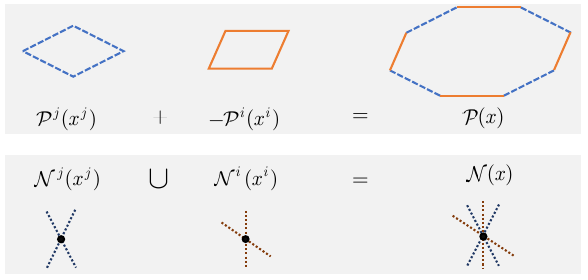}
       \caption{The illustration of the relation between  polygons $\mathcal{P}^j(x^j), -\mathcal{P}^i(x^i), \mathcal{P}(x)$, where $\mathcal{N}^{j}(x^j), \mathcal{N}_{-}^{i}(x^i),\mathcal{N}(x)$ are  their  outer normal vectors. Viewed at the figure, it holds  that  
      $\mathcal{N}(x)= \mathcal{N}^{j}(x^j)\cup \mathcal{N}_{-}^{i}(x^i) $. }
       \label{fig:minkowski_di}
         \vspace{-0.2cm}
    \end{figure}

    \begin{remark}
    It is worth mentioning that the definitions and propositions from \eqref{eq:support-halfspace} to \eqref{eq:mthcalN-ij} are also available for polytopes in $\mathbb{R}^n$, while \eqref{eq:mathcalPij5} holds only for polygons (i.e., polytopes in $\mathbb{R}^2$). 
    This is because the set of all normal vectors of the Minkowski difference $\mathcal{P}(x)=\mathcal{P}^j(x^j)-\mathcal{P}^i(x^i)$ can be given by the union $\mathcal{N}(x)= \mathcal{N}^{j}(x^j)\cup \mathcal{N}_{-}^{i}(x^i) $ when $\mathcal{P}^j(x^j), \mathcal{P}^i(x^i)\subset \mathbb{R}^{2}$, as illustrated in Fig \ref{fig:minkowski_di}. For the case of full-dimensional polytopes $\mathcal{P}^j, \mathcal{P}^i$ in $\mathbb{R}^{n}, n\geq 3$, $\mathcal{N}(x)$ cannot enumerate all normal vectors of the difference $\mathcal{P}(x)$, which has been pointed out by existing references, such as \cite{cameron1986determining, teissandier2011algorithm}. {In 3D case ($n=3$),  enumerating all normal vectors of the difference $\mathcal{P}(x)$ is more complex, however, it is still promising to enumerate all normal vectors by drawing inspiration from the primary  document \cite{cameron1986determining}, the detailed process is left to be studied in the future.} 
     $\hfill\square$ 
    \end{remark}

    Based on {Lemma} \ref{lem:support-halfspace} and {Lemma} \ref{lem:partialK1K2}, {the first main theorem} can be  introduced, which reveals the deeper connection between $h_s(x)$ and $h_a(x)$ in polygon cases.

    \begin{theorem} \label{the:the-signed-distance-is}
    Whenever $\mathcal{P}^i, \mathcal{P}^j$ are polygons, one has: 
    
     1) The inequality in \eqref{eq:mathcalPij4} can be more specified:
    \begin{align} 
    h_{s}(x) \left\{\begin{array}{ll}
    \geq {h}_{a}(x),& ~\text{ if } x \in {\mathcal{S}},  \\ 
    =    {h}_{a}(x),& ~\text{ if } x \in \overline{{\mathcal{S}}^{c}}=
    \partial \mathcal{S}\cup {\mathcal{S}}^{c}. 
    \end{array}
    \right.
    \label{eq:opertsd2}
    \end{align}
   
    2) As a consequence, the following propositions are true: 
    \begin{subequations}\label{eq:underlinesdij}
    \begin{align} 
      {h}_{a}(x)=0  &\Longleftrightarrow x \in  \partial \mathcal{S}, \\
      {h}_{a}(x)>0  &\Longleftrightarrow x \in  \operatorname{Int}(\mathcal{S}),  \\
      {h}_{a}(x)<0  &\Longleftrightarrow x \in  {\mathcal{S}}^{c}. 
    \end{align}
    \end{subequations} 
    \end{theorem}
\begin{proof}
 To avoid breaking the flow of the exposition,  the proof is reported in Appendix \ref{app:proof-them2}, where the required lemma (about distance functions) is given in Appendix \ref{app:df}.
\end{proof}

\begin{remark}
The conclusion \eqref{eq:opertsd2} is not surprising since it has been mentioned in \cite[Sec. V]{cameron1986determining}, where it is explained geometrically.  As far as authors know, it is the first time to give formal proof of the important but non-trivial properties in Theorem \ref{the:the-signed-distance-is}. The fact that $h_a, h_s$ have the same zero-superlevel set, i.e., $\mathcal{S}_{a}=\mathcal{S}$, illustrates that $h_a$ can be regarded as a  non-conservative substitution of $h_s$. It is worth mentioning that the authors do not find Lemma  \ref{lem:support-halfspace} and \ref{lem:partialK1K2} in existing related references, whose proofs are provided in this paper. $\hfill\square$  
\end{remark}
    
The following attempts to figure out whether $h_a$ can be chosen as a barrier function candidate on $\mathcal{S}$. To figure out the differentiability of $h_{a}$, the more detailed expression of $h_{a}$ is shown. 
    
First, recall a well-known property of the Linear Program (LP). From \cite[Sec 2.5, Theorem and Corollary 2]{luenberger1984linear} that the minimum value of a LP can be obtained at the vertexes (extreme points) of the polytope (feasible domain), one knows that LPs in \eqref{eq:matPjPiki} and \eqref{eq:matPjkjPi} can be reduced to 
\begin{align}
  \max_{{\omega}^{j} \in \mathcal{P}^{j}}  (-A^{i}_{k^i})^{\top} {\omega}^{j} &=- \min_{ {l}^{j} \in [r^{j}]} ~  (A^{i}_{k^i})^{\top} v^{j}_{l^j}, \label{eq:minoemgaj}\\
  \max_{{\omega}^{i} \in \mathcal{P}^{i}}  (-A^{j}_{k^j})^{\top} {\omega}^{i} &=- \min_{ {l}^{i} \in [r^{i}]} ~  (A^{j}_{k^j})^{\top} v^{i}_{l^i}, \label{eq:minoemgai}
\end{align}  
where $\{v^{i}_{l^i}: l^i \in  [r^{i}]\}, \{v^{j}_{l^j}: l^j \in  [r^{j}]\}$ are vertex sets of $\mathcal{P}^i, \mathcal{P}^j$ introduced in \eqref{eq:polytope-geometry-2}. Substituting \eqref{eq:minoemgaj}, \eqref{eq:minoemgai}  into \eqref{eq:matPjPiki}, \eqref{eq:matPjkjPi} respectively, and recalling the SDF formula of a point to halfspace in \eqref{eq:opersdvH}, one can get 
\begin{subequations}\label{eq:sd-0-PiPj23}
\begin{align}
    \underbrace{\operatorname{sd}({0},  \mathcal{P}^j - \mathcal{P}^{i}_{k^i})}_{=: \phi^{k^i}(x) }  =&
    \min_{l^j \in [r^{j}]} ~  \underbrace{(A^{i}_{k^i})^{\top}v^{j}_{l^j}  - b^{i}_{k^i}}_{=: \phi^{k^i}_{l^j}(x)},
    \label{eq:sd-0-PjPi2} \\
    \underbrace{\operatorname{sd}({0},  \mathcal{P}^{j}_{k^j} - \mathcal{P}^i )}_{=: \psi^{k^j}(x) }  =&  
    \min_{l^i \in [r^{i}]} ~   \underbrace{(A^{j}_{k^j})^{\top}v^{i}_{l^i}  - b^{j}_{k^j}}_{=: \psi^{k^j}_{l^i}(x)} , 
    \label{eq:sd-0-PiPj3} 
\end{align}
\end{subequations} 
then the explicit expression of $h_{a}$ in \eqref{eq:mathcalPij4} can be given as 
\begin{subequations} \label{eq:underopsd3}
\begin{align} 
    {h}_{a}(x)&= \max \left\{{\phi}(x),  {\psi}(x)\right\},   \\
    {\phi}(x) &:= \max_{k^i \in [r^i]} 
    \phi^{k^i}(x), ~~ \phi^{k^i}(x)= \min_{l^j \in [{r^{j}}]}  
    \phi^{k^i}_{l^j}(x), \\ 
    {\psi}(x) &:= \max_{k^j \in [r^j]} 
    \psi^{k^j}(x), ~ \psi^{k^j}(x)=\min_{l^i \in [{r^{i}}]}  
    \psi^{k^j}_{l^i}(x).
\end{align}
\end{subequations} 
This means that ${h}_{a}(x)$ can be expressed as a three-layer (or two-layer) Boolean logic composition  of atomic component functions $ \phi^{k^i}_{l^j}(\cdot)$, $\psi^{k^j}_{k^i}(\cdot)$,  which is inductively defined with  the following logic syntax \cite{glotfelter2020nonsmooth}:
$$B_{1} \wedge B_{2} := \min\{B_{1},B_{2}\}, B_{1} \vee   B_{2} := \max\{B_{1},B_{2}\}.$$    
Based on the expression given in \eqref{eq:underopsd3}, one can further prove the second main theorem about  that $h_a(x)$ is a NBF candidate on the closed set $\mathcal{S}$ as follows.

\begin{theorem} \label{lem:the-function-sd-candidate-NCBF}
From the Boolean logic composition ${h}_{a}(x)$ in \eqref{eq:underopsd3} and its component functions, one can know that  
     
1) ${h}_{a}(x)$ is smooth composed, i.e., for all ${k^i},{l^i} \in [r^i]$ and ${k^j}, {l^j} \in [r^{j}] $, functions $\phi^{k^i}_{l^j}, \psi^{k^j}_{l^i}$ are smooth. Additionally, their gradients $\nabla \phi^{k^i}_{l^j}$, $\nabla \psi^{k^j}_{l^i}$  never vanish.
    
2) ${h}_{a}(x)$ is a NBF candidate on the closed set $\mathcal{S}$. 
\end{theorem} 
    
\begin{proof}
1) From Assumption \ref{eq:polytope-def} and  definitions in \eqref{eq:sd-0-PiPj23}, functions $ \phi^{k^i}_{l^j}$, $\psi^{k^j}_{l^i}$ are smooth. Applying the chain rule, it yields   
    \begin{subequations}\label{eq:nablaphi}
    \begin{align} 
    \nabla  \phi^{k^i}_{l^j} &=\left[ (\nabla_{x^{i}} \phi^{k^i}_{l^j} )^{\top}, (\nabla_{x^{j}} \phi^{k^i}_{l^j} )^{\top} \right]^{\top}, \\
    \nabla_{x^{i}} \phi^{k^i}_{l^j} &= 
     (\partial A^i_{k^i} )^{\top} 
      v^{j}_{l^j} -  \nabla_{x^{i}}   b^i_{k^i} ,\\
    \nabla_{x^{j}} \phi^{k^i}_{l^j} &= 
      (\partial v^{j}_{l^j})^{\top}  A^i_{k^i} .  
    \label{eq:nablaphi3}
    \end{align}
    \end{subequations}
    where $\partial A^i_{k^i}, \partial v^{j}_{l^j}$ are short for Jacobian matrix 
    $\frac{\partial A^i_{k^i}}{\partial x^i}, \frac{\partial v^{j}_{l^j}}{\partial x^j}$ respectively. 
   From Assumption \ref{eq:polytope-def} again, $\|A^{i}_{k^i}(\cdot)\| = 1$ and $\partial v^{j}_{l^j} (\cdot)$ is of full row rank, then $ \nabla_{x^{i}} \phi^{k^i}_{l^j}$ and $\nabla \phi^{k^i}_{l^j}$  are non-zero vectors.  
    
In a similar way, one can obtain the expression of $\nabla  \psi^{k^j}_{l^i} $: 
    \begin{small}
    \begin{align} 
    \nabla  \psi^{k^j}_{l^i} =[  (\underbrace{ 
    (\partial  v^{i}_{l^i})^{\top}  A^j_{k^j} }_{\nabla_{x^{i}} \psi^{k^j}_{l^i} } )^{\top}, 
     ( \underbrace{  (\partial A^j_{k^j} )^{\top}  v^{i}_{l^i} -  \nabla_{x^{j}}   b^j_{k^j}}_{\nabla_{x^{j}} \psi^{k^j}_{l^i} }
     )^{\top} ]^{\top}, \label{eq:nablapsi}
    \end{align}
    \end{small}
  which also never vanishes since $\nabla_{x^{i}} \psi^{k^j}_{l^i}$ is always non-zero.  
    
2) Since ${h}_{a}(x)$ is smooth composed and due to the existing  $\min/\max$ operators, according to \cite[Proposition IV.9]{glotfelter2020nonsmooth}, ${h}_{a}(x)$ is piecewise $C^1$-continuously differentiable, which means that it is locally Lipschitz continuous.  
Finally, recalling that $\mathcal{S}=\mathcal{S}_{a} $ and the relationships in  \eqref{eq:underlinesdij} in {Theorem} \ref{the:the-signed-distance-is}, ${h}_{a}(x)$ is a barrier function candidate on $\mathcal{S}_{a}$ by verifying Definition \ref{def:NBFcandidate} directly. 
\end{proof}
    
\begin{example}[Continued with Example \ref{exa:rigid_poly}]\label{exa:rigid_poly2} 
Consider the functions $\phi^{k^i}_{l^j}(x),\psi^{k^j}_{l^i}(x)$ in \eqref{eq:sd-0-PiPj23}  where $A^{\mathfrak{i}}_{k^\mathfrak{i}},b^{\mathfrak{i}}_{k^\mathfrak{i}},~\mathfrak{i}\in\{i,j\}$ are defined by \eqref{mathbfHehe}  in Example \ref{exa:rigid_poly}. The expressions of $\nabla  \phi^{k^i}_{l^j},\nabla  \psi^{k^j}_{l^i}$ required in  \eqref{eq:nablaphi}, \eqref{eq:nablapsi} are 
\begin{small} 
$$\begin{aligned}  
    \nabla_{x^{i}} \phi^{k^i}_{l^j}
      &= \left[ -(A^{i}_{k^{i}})^{\top}, \left((v^{j}_{l^{j}}-v^{i}_{k^{i}})^{\top}R(\theta^{i}+\pi) - l_{k^{i}} ^{\top}\right)  \frac{\Delta l_{k^{i}}}{\|\Delta l_{k^{i}}\|}\right]^{\top},  \\
    \nabla_{x^{j}} \phi^{k^i}_{l^j} &= \left[ (A^i_{k^i})^{\top}, (A^i_{k^i})^{\top} R(\theta^i+ {\pi}/{2}) l_{l^j} \right]^{\top} ,  \\
    \nabla_{x^{i}} \psi^{k^j}_{l^i}
      &= \left[ (A^j_{k^j})^{\top}, (A^j_{k^j})^{\top} R(\theta^j+ {\pi}/{2}) l_{l^i} \right]^{\top}, \\
      \nabla_{x^{j}}  \psi^{k^j}_{l^i} &= \left[ -(A^{j}_{k^{j}})^{\top}, \left((v^{i}_{l^{i}}-v^{j}_{k^{j}})^{\top}R(\theta^{j}+\pi) - l_{k^{j}} ^{\top}\right)  \frac{\Delta l_{k^{j}}}{\|\Delta l_{k^{j}}\|}\right]^{\top} .
\end{aligned} $$
\end{small} 
\end{example}

\section{Safety Filter design with a single CBF}\label{sec:safetyfilter}
    
Till now, a Boolean NBF candidate ${h}_{a}(x)$ in \eqref{eq:underopsd3} has been established to approximate the SDF $h_s(x)$. 
The most direct way is to validate that ${h}_{a}(x)$ is further a nonsmooth CBF as in \cite{glotfelter2020nonsmooth}, however, it is challenging, since the nonsmooth analysis would be required and the resulting control laws may be discontinuous. To avoid such complicated mathematical discussions, we attempt to build a smooth approximation of ${h}_{a} (x)$ as a continuously differentiable CBF. Based on this, a safety filter with such a single CBF is established to address the stated problem. 
    
Following in line with the recent work \cite{molnar2023composing}, where a framework is proposed to approximate any nested Boolean composition into a smooth function, the log-sum-exp function is also applied to smooth $h_a$:  
\begin{subequations} \label{eq:underopsd4}
\begin{align} 
  \hat{h}_{a} &:=\frac{1}{\kappa} \ln{  (\hat{\phi}+\hat{\psi} )} 
    - \frac{b}{\kappa},  \\
        \hat{\phi} &:=  \sum_{k^i \in [{r^i}]}  \hat{\phi}^{k^i} , 
    ~~~ \hat{\phi}^{k^i} := \frac{1}{ \mathop{\sum}\limits_{l^j \in [{r^{j}}]} \frac{1}{\hat{\phi}^{k^i}_{l^j}}}, \\
    \hat{\psi}  &:=  \sum_{k^j \in [{r^j}]}  \hat{\psi}^{k^j} , 
      ~~  \hat{\psi}^{k^j} := \frac{1}{\mathop{\sum}\limits_{l^i \in [{r^{i}}]} \frac{1}{\hat{\psi}^{k^j}_{l^i}}},  \\ 
    \hat{\phi}^{k^i}_{l^j} &:=\operatorname{exp}(\kappa \phi^{k^i}_{l^j}),  
    ~~\hat{\psi}^{k^j}_{l^i} :=\operatorname{exp}(\kappa \psi^{k^j}_{l^i}), 
\end{align}
\end{subequations} 
where $\phi^{k^i}_{l^j},\psi^{k^j}_{l^i}$ are component functions of $h_a$ in \eqref{eq:underopsd3}, $\kappa>0$ is a smoothing parameter, and $b\in \mathbb{R}$ is a buffer. Define the zero-superlevel set of $\hat{h}_{a}$ as 
\begin{align*} 
\hat{\mathcal{S}}_{a}:= \hat{\mathcal{S}}_{a}({b,\kappa})=\left\{x: \hat{h}_{a}(x; \kappa,b)\geq 0 \right\}, \label{eq:hatmathcalShat}
\end{align*} 
where $\hat{h}_{a}(x)$ is written as $\hat{h}_{a}(x; b,\kappa)$ sometimes, to highlight its dependence on parameters $b,\kappa$.
     
\begin{lemma} \label{theo:error}
Function $\hat{h}_{a}$ in \eqref{eq:underopsd4} approximates $ {{h}_{a}}$ in \eqref{eq:underopsd3} with the following error bound: 
\begin{equation*}
  {{h}_{a}}(x) - \frac{\ln b_2  + b}{\kappa} \leq   ~\hat{h}_{a}(x; b,\kappa)~ \leq  {{h}_{a}}(x) + \frac{\ln b_1 - b}{\kappa},
  \label{eq:combination_error}
\end{equation*}
where $b_1=r^i+r^j$, $b_2=\max\{r^{i},r^{j}\}$. As a consequence, for fixed $\kappa>0$ and whenever ${b} \geq \ln b_1$, $\hat{h}_{a}(x; {b},\kappa) \leq h_{a}(x)$ holds for all $x$. Furthermore, for any fixed constant $b  \geq \ln b_1$, $\lim_{\kappa \to \infty} \hat{h}_{a}(x; b,\kappa) = {{h}_{a}}(x)$, and $\lim_{\kappa \to \infty}  \hat{\mathcal{S}}_{a}( b,\kappa) = \mathcal{S}_{a}.$
\end{lemma}
    
\begin{proof}
The proof is omitted here for brevity since it is the direct corollary of \cite[Theorem 5]{molnar2023composing}. 
\end{proof}
       
In Lemma \ref{theo:error}, $b\geq \ln b_1$ can guarantee that $\hat{\mathcal{S}}_{a}( b,\kappa)$ is a subset of ${\mathcal{S}}_{a}$. And $\kappa$ should be  sufficiently large to make $\hat{\mathcal{S}}_{a}( b,\kappa)$ be nonempty and close enough to ${\mathcal{S}}_{a}$. As shown in {Theorem} \ref{the:the-signed-distance-is} that  ${\mathcal{S}}_{a}= \mathcal{S}$, it is promising to choose $\hat{h}_{a}(x)$ as a CBF to guarantee the safety w.r.t. $\mathcal{S}$. 
    
To verify this,  applying the chain derivative rule for $\hat{h}_{a}$, one can obtain the explicit expression of $\nabla  \hat{h}_{a}$: 
\begin{align*} 
\nabla  \hat{h}_{a}  &=\frac{1}{\hat{\phi}+\hat{\psi}}  \left(\sum_{k^i\in [r^i]} \nabla \tilde{\phi}^{k^i}  + 
\sum_{k^j\in [r^j]} \nabla \tilde{\psi}^{k^j}\right)  , 
\end{align*}
 where $\tilde{\phi}^{k^i} := \frac{ 1}{\kappa} \hat{\phi}^{k^i},  \tilde{\psi}^{k^j} := \frac{ 1}{\kappa} \hat{\psi}^{k^j}$, and 
\begin{align*} 
\nabla \tilde{\phi}^{k^i} = \frac{\nabla \hat{\phi}^{k^i}}{\kappa} = (\hat{\phi}^{k^i})^2  \sum_{l^j \in [{r^{j}}]} 
\frac{ \nabla \phi^{k^i}_{l^j}   }{\hat{\phi}^{k^i}_{l^j}} , \\
\nabla \tilde{\psi}^{k^j} = \frac{\nabla \hat{\psi}^{k^j}}{\kappa}  = (\hat{\psi}^{k^j})^2  \sum_{k^i \in [{r^{i}}]} \frac{\nabla \psi^{k^j}_{k^i} }{\hat{\psi}^{k^j}_{k^i}}, 
\end{align*}
with $\nabla {\phi}^{k^i}, \nabla {\psi}^{k^j}$ given in \eqref{eq:nablaphi}, \eqref{eq:nablapsi} respectively.
   
Motivated by the CBF condition \eqref{eq:minmathcalLF} and the safety filter design \eqref{eq:QP},  for any given locally Lipschitz continuous nominal controller $u_0=[(u_0^{i})^{\top}, (u_0^{j})^{\top}]^{\top}$,  the following distributed\footnote{The control law \eqref{eq:QPdesign} is called \emph{distributed} instead of \emph{centralized}, due to that  $u^i$ and $u^j$  are  designed separately \cite{funada2024collision}. } QP-based control law is proposed for  input $u$ of system \eqref{eq:dotxfxgxu}:
\begin{subequations}\label{eq:QPdesign}
\begin{align}
    u_{*} = & \underset{u \in \mathbb{R}^{2n}}{\operatorname{argmin}}   ~ \| u - u_{0} \|^2, \label{eq:QPdesign1}  \\
     \quad  \text{s.t. } &  (\nabla_{x^{\mathfrak{i}}} \hat{h}_{a}(x) )^{\top} u^{\mathfrak{i}} \geq -\frac{1}{2}  \alpha(\hat{h}_{a}({x})), \forall 
    \mathfrak{i}\in\{i,j\}, \label{eq:QPdesign2} 
\end{align}
\end{subequations}
where $u= [(u^i)^{\top}, (u^j )^{\top}]^{\top}$  and $\alpha(\cdot)$ is any extended class-$\mathcal{K}$ function. Finally, the polygonal collision avoidance problem is proven to be solved.
    
\begin{theorem}\label{the:conclusion}
Consider the  continuously differentiable  BF candidate  $\hat{h}_{a}(x; b,\kappa)$ on the approximated safety set $\hat{\mathcal{S}}_a(b,\kappa)$ with $b\geq \ln b_1$ and sufficiently large $\kappa>0$. If  $\nabla \hat{h}_{a}(x)$ does not vanish on the boundary of $\hat{\mathcal{S}}_{a}$, i.e.,
\begin{align} 
\nabla \hat{h}_{a}(x)\neq 0, ~\forall x \text{ satisfying } \hat{h}_{a}(x)=0, \label{eq:nabhathax}
\end{align} 
then there exists an open neighborhood of $\hat{\mathcal{S}}_{a}$, named as $\hat{\mathcal{D}}_{a}$, such that:
    
  i) $\hat{h}_{a}(x)$ is a CBF on $\hat{\mathcal{D}}_{a}$ for system \eqref{eq:dotxfxgxu}.
    
  ii) Under the safety filter design in \eqref{eq:QPdesign}, the closed set $\hat{\mathcal{S}}_{a}$ is forward invariant for the resulting closed-loop system $\Sigma$. As a result, $\Sigma$ is also safe w.r.t. $\mathcal{S}$ in \eqref{eq:def-safe-set}.  
\end{theorem}  
    
\begin{proof}
  i). First, \eqref{eq:QPdesign} is obviously feasible whenever $\nabla \hat{h}_{a}(x)$  vanishes in $\operatorname{Int}(\hat{\mathcal{S}}_{a})$. From \eqref{eq:nabhathax} and continuity of $\nabla \hat{h}_{a}$, there exists an open set $\hat{\mathcal{D}}_{a}\supset \hat{\mathcal{S}}_{a}$, $\nabla \hat{h}_{a}(x)$ does not vanish on it. Hence,  the inequality of QP in \eqref{eq:QPdesign} is feasible for all $x\in \hat{\mathcal{D}}_{a}$. As a result, for all $x\in \hat{\mathcal{D}}_{a}$, there exists $u=[(u^{i})^{\top},(u^{j})^{\top}]^{\top}$, such that 
    $$\begin{aligned}
    (\nabla \hat{h}_{a}(x))^{\top}u&=(\nabla_{x^i} \hat{h}_{a}({x}))^{\top} u^{i} + (\nabla_{x^j} \hat{h}_{a}({x}))^{\top} u^{j}\\
    &\geq  \left(- \frac{1}{2} -  \frac{1}{2}\right)  \alpha(\hat{h}_{a}({x})) =  -  \alpha(\hat{h}_{a}({x})).
    \end{aligned}$$
  Referring to Definition \ref{def:cbf},  $\hat{h}_{a}(x)$ is a CBF.
    
    ii) Since $\hat{h}_{a}(x)$ is a smooth CBF, directly from Lemma \ref{lem:set-invariance-and-stability}, 
    the forward invariance of $\hat{\mathcal{S}}_{a}$ can be guaranteed, i.e., $x(t)\in \hat{\mathcal{S}}_{a},\forall t$, whenever 
    $x_0\in \hat{\mathcal{S}}_{a}$. From Theorem \ref{the:the-signed-distance-is} and Lemma \ref{theo:error}, one can know the fact that $\emptyset \neq \hat{\mathcal{S}}_{a}\subset {\mathcal{S}}_{a}={\mathcal{S}}$, which indicates that the resulting closed-loop system is safe w.r.t. $\mathcal{S}$.
    \end{proof}

    \begin{remark}
    Similar to \cite{tan2021high,molnar2023composing}, the optimization-defined control law \eqref{eq:QPdesign} has an explicit solution:
    \begin{subequations}\label{eq:ustari}
    \begin{align}
     &u_{*}^{\mathfrak{i}}  = u_0^{\mathfrak{i}} + \max\{0,\eta^{\mathfrak{i}}(x)\}\frac{\nabla_{x^{\mathfrak{i}}} \hat{h}_{a}(x)}{\|\nabla_{x^{\mathfrak{i}}} \hat{h}_{a}(x)\|^2},~ \forall \mathfrak{i}\in\{i,j\},  \\ 
     &\eta^{\mathfrak{i}}(x) =- (u_{0}^{\mathfrak{i}})^{\top} \nabla_{x^{\mathfrak{i}}} \hat{h}_{a}(x) - \frac{1}{2}  \alpha(\hat{h}_{a}({x})).  
    \end{align}
    \end{subequations} 
    It is worth noticing that condition  \eqref{eq:nabhathax} is a standard condition to formally guarantee the continuity of the QP-based control law \cite{ames2016control, tan2021high}. 
    Although it is proved in {Theorem} \ref{lem:the-function-sd-candidate-NCBF} that the gradients $\nabla \phi^{k^i}_{l^j}$, $\nabla \psi^{k^j}_{l^i}$  never vanish, as pointed out in \cite{lindemann2018control}, using the log-sum-exp softening function could make the gradient of $\hat{h}_{a}$ vanish possibly.  
    $\nabla_{x^{\mathfrak{i}}} \hat{h}_{a}$ vanishing on the boundary will cause the control input $u_{*}^{\mathfrak{i}}$ in \eqref{eq:ustari} to be sufficiently large. For the practical application, this potential problem can be avoided by adding a sufficiently small $\varepsilon>0$ in the denominator, i.e., 
    replacing the term  ``$\|\nabla_{x^{\mathfrak{i}}} \hat{h}_{a}(x)\|^2 $'' with ``$\|\nabla_{x^{\mathfrak{i}}} \hat{h}_{a}(x)\|^2+\varepsilon$''.
    $\hfill\square$  
    \end{remark}

    Till now, all theoretical results have been reported, and their relationships can be concisely found in Fig. \ref{fig:structure}.
    \begin{figure}[!htp]
       \centering
       \vspace{-0.4cm}
       \includegraphics[width=0.7\linewidth]{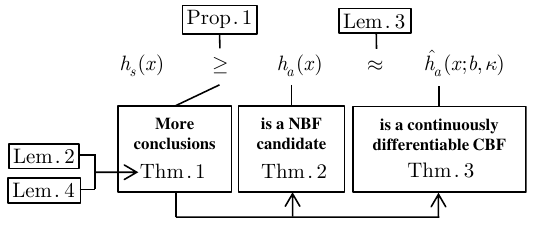}
       \vspace{-0.2cm}
       \caption{The organization of the main theoretical results.}
       \label{fig:structure}
         \vspace{-0.4cm}
    \end{figure}

    In the following, two extensions are discussed. {As pointed out in \cite{lindemann2018control}, CBF proposed for a single integrator can be easily extended to some one-order nonlinear dynamics. For example, 
    the extension of the design \eqref{eq:QPdesign} 
    for the control-affine system $\dot{x}^\mathfrak{i} = f^\mathfrak{i}({x}^\mathfrak{i}) + g^\mathfrak{i}({x}^\mathfrak{i})u^\mathfrak{i} $ is 
    \begin{subequations}\label{eq:QPdesignn} 
    \begin{align}
     u_{*} = & \underset{u \in \mathbb{R}^{2n}} {\operatorname{argmin}} ~ \| u - u_{0} \|^2, \label{eq:QPdesign1n}  \\
      \text{s.t. } & (\nabla_{x^{\mathfrak{i}}} \hat{h}_{a})^{\top}
     ( f^\mathfrak{i} + g^\mathfrak{i} u^\mathfrak{i})\geq -\frac{1}{2}  \alpha(\hat{h}_{a}),  \forall 
    \mathfrak{i}\in\{i,j\},  \label{eq:QPdesign2n}  
    \end{align}
    \end{subequations}
    where  $g^\mathfrak{i} ( g^\mathfrak{i})^{\top}  $ should be positive definite. Such a distributed safety filter is tested for two nonholonomic vehicles in Section \ref{sec:nonholonomic_vehicle}.}  
    On the other hand, the above assumes both $u^i, u^j$ are controlled, while for the case that only $u^i$ is
    controlled and $\dot{x}^{j}$ can be observed, similar safety filters can also be designed, as seen in the underactuated crane example in Section \ref{sec:underactuated_crane}.
  
    \section{Numerical Simulation Examples}\label{sec:numerical-simu}
     In this section, two simulation examples are carried out to show the effectiveness of the above theoretical results. 
    Although the above theoretical results are formally stated for single integrator dynamics, the following two examples show that the proposed method can be generalized to nonlinear dynamics even second-order systems, as in Fig. \ref{Fig:examples}. The video recording of the complete moving process in these simulations can be viewed at  \href{https://youtu.be/KvlB5D3SJtE}{https://youtu.be/KvlB5D3SJtE}.

    \begin{figure}[!htp]
    \vspace{-0.4cm}
       \centering
    \includegraphics[width=0.7\linewidth]{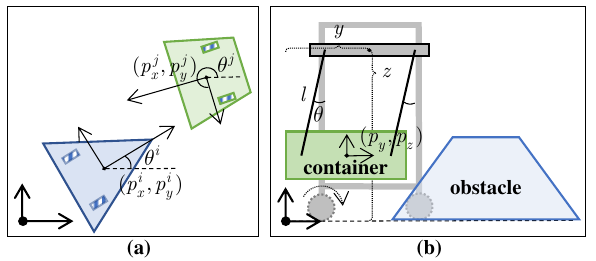}
       \caption{Illustration of the two simulation examples. (a) Coordinated collision avoidance of two polygonal nonholonomic vehicles. (b) Collision avoidance of an underactuated container crane with a polygonal obstacle.}
       \vspace{-0.3cm}
       \label{Fig:examples}
        \vspace{-0.2cm}
    \end{figure}
    
    \subsection{Nonholonomic  vehicle example}\label{sec:nonholonomic_vehicle}
    
    \emph{1) Simulation setup ann primary results}: 
    The efficacy of the proposed approach 
    for coordinated collision avoidance of two polygonal nonholonomic vehicles is demonstrated. Consider the standard nonholonomic kinematic vehicles \cite{dixon2001nonlinear}:
    \begin{equation*}
    \begin{aligned}
    \dot{p}_x^\mathfrak{i} &= v^\mathfrak{i} \cos \theta^\mathfrak{i}, ~
    \dot{p}_y^\mathfrak{i} = v^\mathfrak{i} \sin \theta^\mathfrak{i}, ~
    \dot{\theta}^\mathfrak{i} = \omega^\mathfrak{i}, ~
     \mathfrak{i} \in \{i,j\},
     \end{aligned}
    \end{equation*} 
    with state $x^\mathfrak{i}=[{p}_x^\mathfrak{i},{p}_y^\mathfrak{i},{\theta}^\mathfrak{i}]^{\top}$ and input $ u^\mathfrak{i} =[ v^\mathfrak{i}, \omega^\mathfrak{i}]^{\top}$. 
     
    Consider the situation that vehicles $i,j$  carry a triangular load $\mathcal{P}^{i}(x^{i})$ and a trapezoidal load $\mathcal{P}^{j}(x^{j})$ respectively,  as illustrated in Fig  \ref{Fig:examples}(a). The vertex coordinates of $\mathcal{P}^{i},\mathcal{P}^{j}$ 
    (in their own body frames) are denoted as the columns of  $L^{i},L^{j}$ respectively:
     $$L^{i}=\left[\begin{array}{ccc}  
     3& -2  & -2 \\
     0& -2.5& 2.5
     \end{array}\right], 
      L^{j}=\left[\begin{array}{cccc}  
     1&1&-1&-1 \\
    1.5&-1.5&-1&1 
     \end{array}\right],
     $$ then their $\mathcal{H}, \mathcal{V}$-representations in the inertial frame can be obtained directly from Example \ref{exa:rigid_poly} and the required gradients can be calculated as Example \ref{exa:rigid_poly2}. In this case, approximating 
    polygons as simple circles or ellipses will be too conservative, hence the polygonal collision avoidance problems should be solved directly. 
     
    A situation where potential collisions happen caused by nominal input is introduced to 
    illustrate the role of the safety filter. In detail, the trajectory tracking controller presented in \cite{olfati2002near},  is applied here: 
    $$k^\mathfrak{i}_d = L R(\theta^\mathfrak{i}) (-K^\mathfrak{i} ({p}^\mathfrak{i} - {p}_d^\mathfrak{i}) + \dot{p}_d^\mathfrak{i} ), ~\forall \mathfrak{i}\in\{i,j\},$$
    where $L=[1,0;0,1/l]$ with $l=0.1$, and ${p}^\mathfrak{i}= [{p}_x^\mathfrak{i},{p}_y^\mathfrak{i}]^{\top}$. 
    For simplicity, control gains $K^i, K^j$ are chosen as $I_2$,  while the desired trajectories of vehicles $i,j$ are set  respectively as 
    $$\begin{aligned}
     {p}_d^i &=[18\sin(-0.8t-0.5\pi), 12\cos(-0.8t-0.5\pi)]^{\top}, \\
     {p}_d^j &=[14\sin(~~0.9t+0.5\pi), ~~9\cos(~~0.9t+0.5\pi)]^{\top}.
    \end{aligned}$$
    Accordingly, the initial pose of vehicles $i,j$ are set to 
    $x^{i}(0)=[-18,0, -0.5\pi]^{\top}$ and $
     x^{j}(0)=[ 14,0,  0.5\pi]^{\top}$ (with units $\mathrm{m},\mathrm{m}, \mathrm{rad}$).
    
    \begin{figure}[!htp]
    \vspace{-0.2cm}
       \centering
    \includegraphics[width=0.7\linewidth]{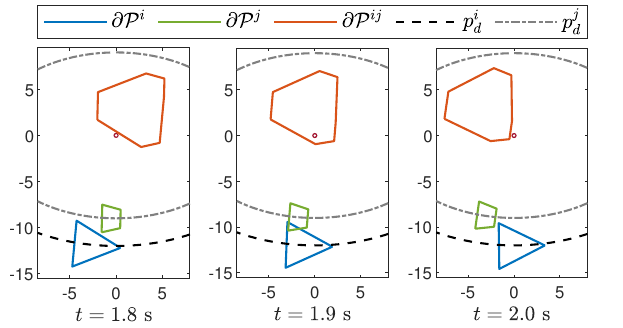}
       \caption{Illustration of the moving process of vehicles $i,j$ using the nominal control law (without the safety filter). By observing that  $\mathcal{P}^{i}$ (blue region)  intersects $\mathcal{P}^{j}$  (green region), as well as the origin (red dot) is inside their Minkowski difference $\mathcal{P}^{ij}=\mathcal{P}^{j}- \mathcal{P}^{i}$ (red region), one can know that the collision happens around $1.9~\mathrm{s}$.}
       \label{fig:Sim_s0}
      \vspace{-0.6cm}
    \end{figure}
    
   The simulations are all performed in Matlab/Simulink. 
    Fig. \ref{fig:Sim_s0} shows the moving process of these two vehicles with the nominal inputs, where the collision happens around $1.9~\mathrm{s}$. 
    {Fig. \ref{Sim_s2}  illustrates the moving process of vehicles with the proposed distributed safety filter \eqref{eq:QPdesignn} using different choices of $\kappa$.} 
    In order to explain this matter, the function 
    $\alpha(\cdot)$ is set to $\alpha(h)=5h,\forall h$ for simplicity, and  
    the gain $b$ is picked up as $b= \ln(3+4)$ according to {Lemma} \ref{theo:error}. The results show that $\kappa=5$ can produce more aggressive collision avoidance behavior than $\kappa=1$, which is consistent with 
    the theoretical result obtained in \emph{Lemma} \ref{theo:error}.  
    
    \begin{figure}%[htb]   
      \centering
      \subfloat[$\kappa=5$, producing aggressive avoidance behavior.]{\includegraphics[width=0.34\textwidth]{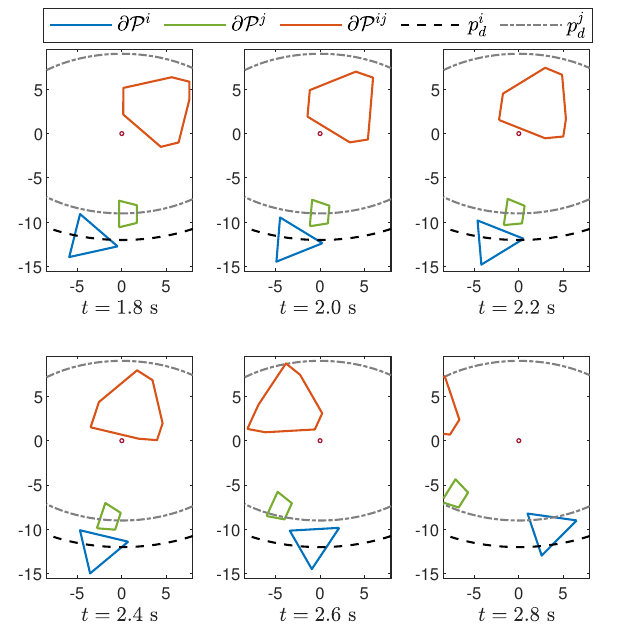}}

      \subfloat[$\kappa=1$, producing conservative avoidance behavior.]
      {\includegraphics[width=0.34\textwidth]{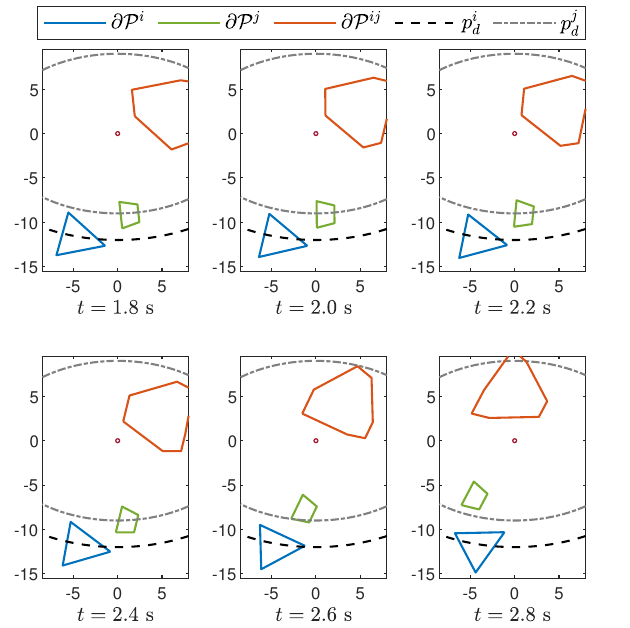}}
    \caption{ Illustration of the moving process of vehicles $i,j$ under the proposed safety filter in  X-Y plane  with different parameters for $\kappa$.  
    } \label{Sim_s2}   
     % \vspace{-0.2cm}  
    \end{figure}  
    
    For clarity, the time response of the value of CBF $\hat{h}_{a}(x(t)) $ at the time interval $[0,10]~ (\mathrm{s})$ is shown in Fig. \ref{fig:Sim_s0_comp}. From it, one can observe that
    the red line is below zero sometimes, which means that the collision between vehicles happens when the safety filter (SF) is not applied. By comparing the evolution of the blue and green lines, one can conclude that the SF with $\kappa=1$ will make a faster response to collision
    than the SF with $\kappa=5$. This confirms the conclusion obtained in the last paragraph that the higher $\kappa$ can produce more aggressive collision avoidance behaviors.
     
    \begin{figure}%!htp]
         \centering
      \includegraphics[width=0.7\linewidth]{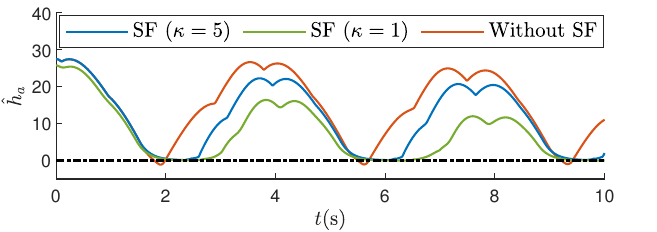}
        \vspace{-0.2cm}
         \caption{The time responses of $\hat{h}_{a}(x(t)) $ at time interval $[0,10]~ (\mathrm{s})$. 
         It shows that the collision happens when the SF is not applied (seen from the red line) and the SF with $\kappa=1$ (the green line)  makes an earlier response to collision
      than the SF with $\kappa=5$ (the blue line).  }
         \label{fig:Sim_s0_comp}
          \vspace{-0.6cm}
      \end{figure}
    
    {\emph{2) Comparison results}: To further illustrate the effectiveness of the proposed method, the comparison with the existing SDF-linearized CBF method \cite{singletary2022safety}  is provided.   In detail, $\kappa$ is chosen as $5$, and other parameters of opti-free CBF are the same as the ones in \emph{Subsection 1)}. Then  the SDF-linearized CBF method (under different sampling densities, i.e., sampling each edge by $10, 15, 20$ points respectively) is chosen as  a method of comparison. 
    %Limited by pages, details of  comparison settings can be found in the complete edition of this paper \cite{wu2025optimization}.} 
     
     By the aid of  \cite[Eq. (26)]{singletary2022safety}, $h_s(x)$ can be rewritten as 
    \begin{align}
    h_s(x)
    =& \max_{a^{\top} a=1}~  \min _{ {
    \omega^i \in \mathcal{P}^i(x^i),  
    \omega^j \in \mathcal{P}^j(x^j)}} ~ a^{\top} (\omega^j - \omega^i) \label{eq:hsxmax} \\
      =& ~ \hat{a}(x)^{\top}  (\underbrace{{p}^{j}+{R}(\theta^{j})\hat{g}^{j}(x)}_{=\hat{\omega}^j(x)} - \underbrace{({p}^{i}+{R}(\theta^{i}) \hat{g}^{i}(x))}_{=\hat{\omega}^i(x)}  ).   \nonumber
    \end{align} 
    Here, vectors $\hat{a}$  and $\hat{\omega}^i, \hat{\omega}^j$ denote the direction and points
    that maximize and minimize the expression in \eqref{eq:hsxmax} respectively. These  vectors depend on  $x=\operatorname{col}(x^i,x^j)$ with $x^i=\operatorname{col}(p^i_x, p^i_y,\theta^i)$, $x^j=\operatorname{col}(p^j_x, p^j_y,\theta^j) $.  Referring to \cite[Eq. (27)]{singletary2022safety},  gradients $  \nabla_{x^i}  h_s(x),  \nabla_{x^j}  h_s(x) \in \mathbb{R}^{3\times1}$ can be approximated by the local linearization: 
    %\begin{subequations}
    \begin{align*}
    \nabla_{x^i}  h_s(x) &\approx  -[ \hat{a}(x)^{\top}, \hat{a}(x)^{\top}R(\theta^{i}+ {\pi}/2 ) \hat{g}^{i}(x) ]^{\top} , \\
    \nabla_{x^j}  h_s(x) &\approx  ~~[ \hat{a}(x)^{\top},  \hat{a}(x)^{\top}R(\theta^{j}+ {\pi}/2 ) \hat{g}^{j}(x) ]^{\top} ,
    \end{align*}
  %\end{subequations}
    where  $\hat{a}(x)=\frac{\hat{\omega}^j(x) - \hat{\omega}^i(x)}{\|\hat{\omega}^j(x) - \hat{\omega}^i(x)\| }$,  and 
    $$
    \hat{g}^{i}(x)= {R}(\theta^{i})^{\top} ( \hat{\omega}^i(x) - {p}^{i} ), ~ \hat{g}^{j}(x)= {R}(\theta^{j})^{\top} ( \hat{\omega}^j(x) - {p}^{j} ). 
    $$ 
    It is worth noting that calculating (or approximating) the nearest points $\hat{\omega}^i(x), \hat{\omega}^j(x)$ is a 
    time-consuming task,  regardless of applying the enumeration/ergodic  search, well-known Gilbert–Johnson–Keerthi (GJK) algorithm or other collision detection methods. Here, to balance the difficulties in implementation, the enumeration/ergodic search is applied to approximate $\hat{\omega}^i(x), \hat{\omega}^j(x)$ at different $x$   by sampling the boundaries of two polygons properly:
    $$
    [\hat{\omega}^i(x)^{\top}, \hat{\omega}^j(x)^{\top}]^{\top} \approx  \argmin_{ {\omega}^i \in \partial \tilde{\mathcal{P}}^{i}(x), {\omega}^j \in \partial \tilde{\mathcal{P}}^{j}(x) }  \|  {\omega}^i - {\omega}^j    \|, 
    $$
    where finite point sets $\partial \tilde{\mathcal{P}}^{i}, \partial \tilde{\mathcal{P}}^{j}$  are sampled boundaries  of two polygons. In detail, 
    each edge of these two polygons is uniformly sampled by $10, 15, 20$ points respectively. The  SDF-linearized CBF $h_s(x)$  and $\nabla  h_s(x)$ can be calculated by the aid of the above approximations, to substitute the opti-free CBF $\hat{h}_a(x)$ and $\nabla  \hat{h}_a(x)$. 
    
Then, the SDF-linearized CBF is also simulated in Matlab/Simulink, and the comparison result is reported as follows. On the one hand, the system response under SDF-linearized CBF methods (under different sampling densities) and the proposed one are compared in Fig. \ref{fig:Sim3_comparisons}, where the function $h_a(x)$ is applied to measure whether collisions happen. From the results, when the sampling density is $15$ or $10$, the SDF-linearized CBF may lead to a collision behavior, which is illustrated by the red or green lines. When the sampling density is up to $20$, the SDF-linearized CBF,  as well as the proposed method, can avoid the collision between vehicles, as seen from the black and blue lines. On the other hand, the computing time  of these methods can be obtained by the aid of Simulink/Profiler directly, and they are compared in Tab. \ref{fig:Comparisons_little}. From the results, one can know that  the proposed method is several times  faster than the SDF-linearized CBF, no matter the sampling density of the SDF-linearized CBF is chosen as  $10, 15$ or $20$.  

As a brief summary,  both methods could avoid the collision, but the proposed method is more efficient than the SDF-linearized CBF,  which is a typical example of implicit, optimization-embedded methods in Tab. \ref{tab:comparisons}.

\begin{figure}[!htp]
      \vspace{-0.2cm}
         \centering
      \includegraphics[width=0.8\linewidth]{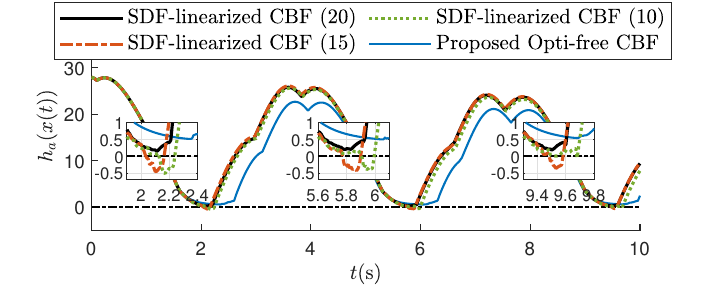}
      \vspace{-0.2cm}
         \caption{{Comparisons of the responses of $h_{a}(x(t))$, where function $h_a(x)$ is applied to measure whether collisions happen.  when the sampling density is $15$ or $10$, the SDF-linearized CBF may lead to a collision behavior, which is illustrated by the red or green lines. When the sampling density is up to $20$, 
         the SDF-linearized CBF,  as well as the proposed method, can avoid the collision between vehicles, as seen from the black and blue lines.}}
         \label{fig:Sim3_comparisons}
       \vspace{-0.3cm}
\end{figure}
      
\begin{table}[]
      % \vspace{-0.2cm}
         \centering
      \includegraphics[width=0.8\linewidth]{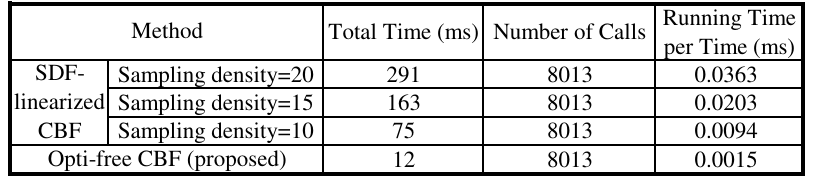}
    \vspace{-0.2cm}
       \caption{Comparisons of the computing time. Running Time per Time is calculated by dividing the Total Time (TT) by the Number of Calls (NoC), where  TT and NoC are directly collected from Simulink/Profiler directly. From the above, one can observe that the proposed Opti-free CBF method is more efficient than the SDF-linearized CBF method (under different sampling densities). }  
         \label{fig:Comparisons_little}
\end{table}
    
\subsection{Underactuated crane example}\label{sec:underactuated_crane}
This subsection shows that the proposed smooth approximated signed distance function $\hat{h}_{a}$ could be extended to avoid the collision between a container crane and a moving polygonal obstacle.
      
\emph{1) {Dynamics} and geometry setup}: For simplicity, we consider the following situation. A rectangular container $\mathcal{P}_{i}$ is attached to the cart by two parallel ropes, and a trapezoidal obstacle $\mathcal{P}_{j}$ moves on the land, as shown in Fig \ref{Fig:examples}(b). The configuration variable of this system is $q=\left[y(t), z(t) , \theta(t)\right]^{\top}$, and the {dynamics} of such an underactuated container crane is similar to an aerial pendulum, expressed as the following underactuated Euler-Lagrangian equation:
\begin{equation}\label{eq:contrainer_crane}
\begin{aligned}
      (M+m) \ddot{y}+m l\left(\ddot{\theta} \cos \theta-\dot{\theta}^2 \sin \theta\right) &=u_y, \\
       (M+m)(\ddot{z}+g)+m l\left(\ddot{\theta} \sin \theta+\dot{\theta}^2 \cos \theta\right)&=u_z, \\
      m l \ddot{y} \cos \theta+m l \ddot{z} \sin \theta+m l^2 \ddot{\theta}+m g l \sin \theta  &=0,
      \end{aligned}
\end{equation} 
with system parameters $M=10~ (\mathrm{kg}), m=5~ (\mathrm{kg}), g=9.8~ (\mathrm{N}\cdot \mathrm{kg}^{-1})$ and $l=0.7~ (\mathrm{m}) $.  
In fact, this dynamic is consistent with the Y-Z dynamics of the outer-loop system of the aerial transportation system \cite{yu2024fault, liang2019novel}. 

From the geometric relation shown in Fig \ref{Fig:examples}(b), the center of the container can be expressed as   
\begin{equation}\label{eq:pyz}
    p^i=[{p}_y,{p}_z]^{\top}=[y+l \sin \theta,z-l\cos \theta ]^{\top}. 
\end{equation}
Besides, the center of the obstacle is denoted as $p^j\in\mathbb{R}^2$. The attitudes of $\mathcal{P}^{i},\mathcal{P}^{j}$ are always identity matrices, and their vertexes are given by the columns of the following matrices: 
$$\begin{aligned}
     L^i &=\left[\begin{array}{cccc}  
     0.5 &  0.5& -0.5   &-0.5 \\
     0.25& -0.25& -0.25 &0.25
     \end{array}\right], \\
     L^j &= \left[\begin{array}{cccc}  
    1&3& -3& -1 \\
    0.5& -0.5& -0.5& 0.5 
     \end{array}\right],
\end{aligned}$$
based on which the $\mathcal{H}, \mathcal{V}$-representations of $\mathcal{P}^{i}(p^i), \mathcal{P}^{j}(p^j)$ in the initial frame 
can be obtained from Example \ref{exa:rigid_poly}. 
    
\emph{2) Nominal controller and safety filter design}:  A situation is considered here where the container crane is moving against the obstacle. The following PD-like stabilization control law \cite{liang2019novel} is applied to the input $u=[u_y,u_z]^{\top}$, so as to drive the crane  to the desired configuration  $q_d=[y_d,z_d,0]^{\top}$:
\begin{equation}\label{eq:u_0Kp}
    u_{0}=-K_p \operatorname{tanh}(\chi - \chi_d) - K_d \operatorname{tanh}(\dot{\chi}) + (M+m)ge_2, 
\end{equation}
where $\operatorname{tanh}(u):=[\operatorname{tanh}(u_1),\operatorname{tanh}(u_2)]^{\top},\forall u=[u_1,u_2]^{\top}$, and 
$$\begin{aligned}
    \chi &=[y+\lambda \sin\theta, z-\lambda \cos\theta]^{\top},
    \quad \chi_d =[y_d, z_d-\lambda ]^{\top}, \\
    \dot{\chi} &=[\dot{y}+\lambda \dot{\theta}\cos\theta, \dot{z}+\lambda \dot{\theta}\sin\theta]^{\top},\quad e_2=[0,1]^{\top}. \\
    % \chi_d &=[y_d+\lambda \sin(0), z_d-\lambda \cos(0)]^{\top}=[y_d, z_d-\lambda ]^{\top} ,
    \end{aligned}$$
Control gains are set to $K_p=[5,0;0,10],K_d=[7,0;0,7]$, and $\lambda=-0.01$.  Finally, the velocity of the obstacle is set to $\dot{p}^j=[-0.1,0]^{\top} (\mathrm{m/s})$ with initial position $p^j(0)=[6,0.5]^{\top}  (\mathrm{m})$. On the other hand, the initial and desired 
configuration values of the crane are set to $q(0)=[0,1.5,0]^{\top}$ and $q_d=[10,1.5,0]^{\top}$ (with units $\mathrm{m},\mathrm{m}, \mathrm{rad}$) respectively, while the initial/desired velocities are zeros. Under the above setting, the moving process of the container crane and obstacle is shown in Fig. \ref{Sim2}(a).
    
Now, the proposed function $\hat{h}_{a}(p^i,p^j)$ will be extended to achieve collision avoidance between such a container crane and obstacle. Specifically, substituting $p^i=[p_y,p_z]^{\top}$ defined in \eqref{eq:pyz} into \eqref{eq:contrainer_crane}, one can obtain the equivalent dynamical equation under the new  coordinate $p=[(p^i)^{\top},\theta]^{\top}$: 
\begin{equation}\label{eq:new_expression}
M_{T}\ddot{p}+C_{T}\dot{p}+G_{T}=B_{T} u,
\end{equation}
where %expressions of $M_{T}, C_{T},G_{T}$ and $B_{T}$ are shown in \cite{wu2025optimization}.
 \begin{small}
    $$
    \begin{aligned}
    M_{T} &=\left[\begin{array}{ccc}
    M+m & 0 & -M l \cos \theta \\
    0 & M+m & -M l \sin \theta \\
    -M l \cos \theta & -M l \sin \theta & M l^2
    \end{array}\right], \\
    C_{T}&=\left[\begin{array}{ccc}
    0 & 0 &   M l \dot{\theta} \sin \theta \\
    0 & 0 & - M l \dot{\theta} \cos \theta \\
    0 & 0 & 0
    \end{array}\right],   \\
    {G}_{T} &= 
    \left[\begin{array}{c}
      0\\
      (M+m) g\\
      -M g l \sin \theta
      \end{array}\right], 
    {B}_{T} =
      \left[\begin{array}{cc}
        1 & 0  \\
        0 & 1  \\
        -l\cos\theta & -l\sin\theta
        \end{array}\right]. 
    \end{aligned}
    $$
    \end{small}

Inspired by the energy-based CBF construction method for full-actuated systems in \cite{singletary2021safety}, one could extend  $\hat{h}_{a}(p^i,p^j)$ proposed in \eqref{eq:underopsd4} to build a CBF for the underactuated system \eqref{eq:new_expression} or \eqref{eq:contrainer_crane}: $$\phi(p,\dot{p},p^j):= \eta \hat{h}_{a}(p^i,p^j) - \frac{1}{2} \dot{p}^{\top}  M_{T}(\theta) \dot{p},$$
where  $p=[(p^i)^{\top},\theta]^{\top}$ and $\eta>0$ is a user-defined gain. Its time derivative along the system is 
$$ \begin{aligned}
    &\dot{\phi}(p,\dot{p},p^j,u) =  \eta  \dot{\hat{h}}_{a}(p^i,p^j)  - \frac{1}{2}  \dot{p}^{\top}  \dot{M}_{T}(\theta) \dot{p}
    -   \dot{p}^{\top}  {M}_{T}(\theta) \ddot{p}
    \\
    &= \eta  \dot{\hat{h}}_{a}(p^i,p^j)  
     - \dot{p}^{\top}\left(B_{T}u - G_T  +\frac{1}{2}(\dot{M}_{T}-2 C_{T}) \dot{p}\right) \\
     &=-[\dot{y}^{\top}, \dot{z}^{\top}] u + \eta ((\nabla_{p^i} \hat{h}_{a})^{\top}\dot{p}^{i} +
    (\nabla_{p^j} \hat{h}_{a})^{\top}\dot{p}^{j} ) 
    + \dot{p}^{\top}  G_T, 
  \end{aligned}$$
where the property is applied that  $(\dot{M}_{T}-2 C_{T})$ is an anti-symmetric matrix.

Then the following safety filter can be applied for $u$ to modify the nominal input $u_0$ in \eqref{eq:u_0Kp} to produce the collision avoidance behavior: 
$$\begin{aligned}
     u_{*} =  \underset{u \in \mathbb{R}^{2}} {\operatorname{argmin}} &~  ( u - u_{0})^{\top}Q ( u - u_{0}),  \\
    &  \text{ s.t. }  
    \dot{\phi}(p,\dot{p},p^j,u)
     \geq -  \alpha \phi(p,\dot{p},p^j), 
\end{aligned}$$
where the weighting matrix $Q=[1000, 0;0, 2]$  and other gains are set as $\eta=500,\alpha=3,\kappa=5,b=8$. With this safety filter, the moving process is shown in Fig. \ref{Sim2}(b). 
    
\emph{3) Results}: The simulations are also performed in Matlab/Simulink. Fig. \ref{Sim2} illustrates the process of the container crane and the moving obstacle in the Y-Z plane. In subfigure (a), driven by the nominal controller, there is a sustained collision between the container $\mathcal{P}^i$ (blue line) and the obstacle $\mathcal{P}^j$ (green line) during the time interval $[6,11] ~(\mathrm{s})$. In subfigure (b),  such collision is avoided with the safety filter. \\
     
In summary, such two examples show that the proposed smooth approximated signed distance function $\hat{h}_{a}$ can be extended easily and has great potential to achieve polygonal collision avoidance for complex nonlinear systems, such as nonholonomic vehicles and underactuated Euler-Lagrangian dynamics.

\begin{figure}%[htb]   
      \centering
      \vspace{-0.2cm} 
      \subfloat[Response of the nominal controller without the safety filter. There is a sustained collision between the container $\mathcal{P}^i$ and the obstacle $\mathcal{P}^j$ from $6~(\mathrm{s}) $ to $11~(\mathrm{s}) $.]{\includegraphics[width=0.37\textwidth]{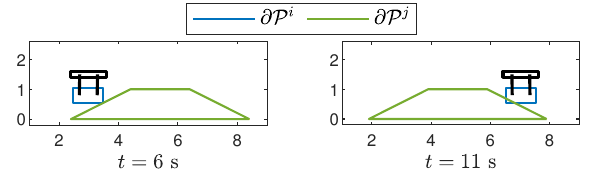}}
      
      \subfloat[Response of the nominal controller with the safety filter. With the safety filter,  
      such collision is avoided.]{\includegraphics[width=0.39\textwidth]{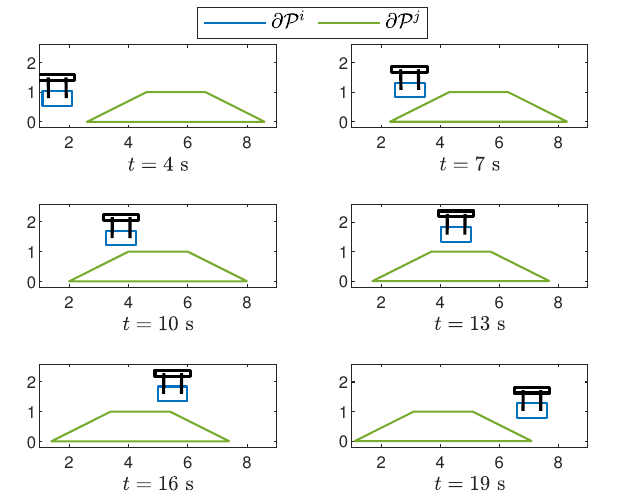}} 
    
      % \subfloat[Response of the nominal controller with the safety filter. With the safety filter,  
      % such collision is avoided.]{\includegraphics[width=0.39\textwidth]{figures/Sim2_s1_kappa5}} 
    \caption{Illustration of the process of container crane and the moving obstacle in the Y-Z plane without/with the proposed safety filter. 
    } \label{Sim2}   
      \vspace{-0.2cm}  
\end{figure}  
            
\section{Conclusion}
\label{sec:conclusion}
In this paper, a general framework is presented for the obstacle avoidance between polygons by an optimization-free smooth CBF.
This is achieved by proposing a smooth approximation of a lower bound of the traditional SDF, where great efforts have been made to study and guarantee its non-conservatism in theory. 
Simulations are carried out for two complex nonlinear systems, to show  
 wide application scenarios of the above theoretical results. 
{Since the proposed CBF method is optimization-free, resulting in a very high calculation speed, it is very suitable to be extended for large-scale systems, such as coordinated collision avoidance of  polytopic-shape multi-robots under cluttered environments. The biggest limitation of the proposed method is that it is now only applicable to two-dimensional cases, while how to extend it to the 3D situation is not  obvious. Besides, input constraints are not considered, hence the proposed method may fail to work when the control input is constrained.}  
Future works involve doing experimental verifications on real-world robotic platforms, extending to  {more general polytopic multi-robots, high-order dynamics, control input constraints, and more complex polytopic timed reach-avoid specifications}. 
 
\bibliographystyle{IEEEtran}
% \balance
\bibliography{bibliography}

\appendix
\setcounter{equation}{0}
\renewcommand{\theequation}{A.\arabic{equation}}

\subsection{Proof of {Lemma} \ref{lem:support-halfspace}}
\label{app:proof-lem2}

\begin{proof}
First, recall the properties of support function in \cite{schneider2014convex} that 
for any convex sets $K_1, K_2$ and vector $a$, one has 
\begin{subequations}\label{eq:hK12a}
\begin{align}
h_{K_{1}+K_{2}}(a) &= h_{K_{1}}(a) + h_{K_{2}}(a), \\
h_{K_{1}-K_{2}}(a) &= h_{K_{1}}(a) + h_{-K_{2}}(a), \\
h_{-K_{2}}(a) & = h_{K_{2}}(-a). 
\end{align}
\end{subequations}
The Eq. \eqref{eq:hK12a} means that the support function of a Minkowski sum/difference is the sum/difference of the support functions of components. Based on \eqref{eq:hK12a}, one can obtain
\begin{align}
h_{\mathcal{P}}(-A^{i}_{k^i}) 
&= h_{\mathcal{P}_j}(-A^{i}_{k^i}) + h_{\mathcal{P}_i}(A^{i}_{k^i}) 
\nonumber \\
&=\max_{{\omega}^{j} \in \mathcal{P}^{j} } (-A^{i}_{k^i})^{\top} {\omega}^{j} +
  \underbrace{ \max_{{\omega}^{i} \in \mathcal{P}^{i} } ( A^{i}_{k^i})^{\top} {\omega}^{i}}_{=b^{i}_{k^i}} , \label{eq:hmpAiki} \\
  h_{\mathcal{P}}(A^{j}_{k^j}) 
&= h_{\mathcal{P}_j}(A^{j}_{k^j}) + h_{\mathcal{P}_i}(-A^{j}_{k^j}) 
\nonumber  \\
&= \underbrace{\max_{{\omega}^{j} \in \mathcal{P}^{j} } (A^{j}_{k^j})^{\top} {\omega}^{j} }_{=b^{j}_{k^j}}
   + \max_{{\omega}^{i} \in \mathcal{P}^{i} } (- A^{j}_{k^j})^{\top} {\omega}^{i}.   \label{eq:hmpAjkj}
\end{align}
By recalling the definition of support halfspace in \eqref{eq:support-halfspace}, and observing \eqref{eq:matPjPiki}, \eqref{eq:hmpAiki}, as well as \eqref{eq:matPjkjPi}, \eqref{eq:hmpAjkj}, one can know that 
\eqref{eq:mathcalPkijwu} holds.   

Next, we attempt to prove \eqref{eq:mathcalPij5}. In particular, the second equality in \eqref{eq:mathcalPij5} can be obviously obtained based on \eqref{eq:mathcalPkijwu} and \eqref{eq:mthcalN-ij}, hence one only needs to prove the first equality in \eqref{eq:mathcalPij5}. 
Working towards this, recall a classical result \cite[Corollary 2.4.4]{schneider2014convex}:  
for any polytope $\mathcal{Q}$, let $\mathcal{A}= \{a_i: i=1,2,...,k\}$ denote  the set of all normal vectors of $\mathcal{P}$, then the polytope $\mathcal{Q}$ can be expressed as the intersection of support halfspaces corresponding to  normal vectors in $\mathcal{A}$, i.e., 
$\mathcal{Q}=\bigcap_{a\in \mathcal{A}} H_{\mathcal{Q}}^{-}(a)$. 
Motivated by this proposition, the first equality in \eqref{eq:mathcalPij5} can be obtained directly if 
$\mathcal{N}(x)$ can be proved to be the set of all normal vectors of $\mathcal{P}(x)$. 

In fact, for polygons $\mathcal{P}^j(x^j), \mathcal{P}^i(x^i)\subset \mathbb{R}^{2}$, $\mathcal{N}(x)$ is truly the set of all normal vectors of $\mathcal{P}(x)$, as seen in  Fig \ref{fig:minkowski_di}. In more detail, as pointed out in \cite[Section 3.1]{teissandier2011algorithm}, 
the edges of the polygonal (Minkowski) addition are translations of the edges of the two operand polygons. Similarly, the normal vector of the polygonal difference $\mathcal{P}=\mathcal{P}^j-\mathcal{P}^i$ are translations of the normal vectors of the two operand polygons $\mathcal{P}^j, -\mathcal{P}^i$. That is to say, the union $\mathcal{N}= \mathcal{N}^{j}\cup \mathcal{N}_{-}^{i}$ is the set of all normal vectors of the Minkowski difference $\mathcal{P}=\mathcal{P}^j - \mathcal{P}^i$. 
\end{proof}

\subsection{Proof of {Theorem} \ref{the:the-signed-distance-is}}
\label{app:proof-them2}
    
\begin{proof}
1) In fact,  ``$\geq$'' in \eqref{eq:opertsd2} has been shown in {Proposition} \ref{lem:hsxhax}. Hence, one only needs to prove that ``$=$'' in \eqref{eq:opertsd2} holds on the mentioned point set,  as proved in the following two steps. 
   
Step 1: Recalling \eqref{eq:mathcalPij5} in {Lemma} \ref{lem:support-halfspace}, $\mathcal{P}$ can be written as 
   \begin{align}
   \mathcal{P}=\bigcap_{k\in [r^i+r^j]} \mathcal{Q}_{k}
   =\bigcap_{k\in [r^i+r^j]} \left(\mathcal{Q}_{k} \cap \widetilde{\mathcal{Q}}_{k} \right), \label{eq:mathcalPij-21} 
   \end{align}
    where 
   $
   \widetilde{\mathcal{Q}}_{k}= \bigcap_{k\neq m\in [r^i+r^j]} \mathcal{Q}_{m} 
   $. 
   Based on \eqref{eq:mathcalPij-21} and Lemma \ref{lem:partialK1K2} in Appendix \ref{app:df}, the following relations hold:  
   % \begin{subequations}
   % \begin{align} 
   % \partial \mathcal{P} 
   % & =  \bigcup_{k^i\in [r^{i}] } \left( \partial  \left(\mathcal{P}^j -   \mathcal{P}^{i}_{k^i}\right) \bigcap \left( %  
   % \bigcap_{m^i \neq  k^i } \left(\mathcal{P}^j -   \mathcal{P}^{i}_{m^i}\right)\right) \right) \\
   % & =\bigcup_{k^j\in [r^{j}] } \left( \partial  \left(   \mathcal{P}^{j}_{k^j}- \mathcal{P}^i\right) \bigcap \left(   
   % \bigcap_{m^j \neq  k^j } \left(\mathcal{P}^{j}_{m^j}- \mathcal{P}^i\right)\right) \right), 
   % \end{align}
   % \begin{align} 
   %
   % \end{align}
   % \end{subequations}
   \begin{subequations} \label{eq:partial-Pijbig}
   \begin{align} 
    \partial \mathcal{P} 
   &=   
   \bigcup_{k\in [r^i+r^j]} \left( \partial  \mathcal{Q}_{k} \cap \widetilde{\mathcal{Q}}_{k} \right), 
   \\
    d( {0},\partial \mathcal{P}) 
    & =  \min_{k\in [r^i+r^j]} d( {0},  \partial  \mathcal{Q}_{k} \cap \widetilde{\mathcal{Q}}_{k} ) .  
     \label{eq:d-0-partial-Pij}
   \end{align}
   \end{subequations}
    and for every $k\in [r^i+r^j] $, one has 
   \begin{subequations} \label{eq:partial-mathcalPij}
   \begin{align} 
    \partial  \mathcal{Q}_{k} 
   & = ( \partial  \mathcal{Q}_{k} \cap \widetilde{\mathcal{Q}}_{k})  \cup ( \partial  \mathcal{Q}_{k} / \widetilde{\mathcal{Q}}_{k}), \\
   d( {0}, \partial  \mathcal{Q}_{k}  ) 
   & = \min   \{   d( {0},  \partial  \mathcal{Q}_{k} \cap \widetilde{\mathcal{Q}}_{k} ), 
    d( {0},  \partial  \mathcal{Q}_{k} / \widetilde{\mathcal{Q}}_{k})  \}.  
    \label{eq:d-0-partial-wid-Pij} 
   \end{align}
   \end{subequations}
    
Step 2: Recalling the expressions of closed sets $\mathcal{P}$ in \eqref{eq:mathcaiPij} and $\mathcal{S}$ in \eqref{eq:def-safe-set}, one can obtain: 
\begin{equation}\label{eq:xmathSijc}
\begin{aligned} 
   x &\in   {\mathcal{S}}^{c} 
   &\Leftrightarrow~  h_{s}(x)<0
   \Leftrightarrow&~  {0} \in \operatorname{Int}  (\mathcal{P}(x)), \\
   x &\in  \partial \mathcal{S} 
   &\Leftrightarrow~ h_{s}(x)=0
   \Leftrightarrow&~ {0} \in \partial \mathcal{P}(x) , \\
   x&\in  \operatorname{Int}(\mathcal{S}) 
   & \Leftrightarrow~ h_{s}(x)>0 
   \Leftrightarrow&~  {0} \in  (\mathcal{P}(x))^{c}.  
\end{aligned}
\end{equation}
When $x \in \partial \mathcal{S} \cup {\mathcal{S}}^{c}$, one has $ {0} \in \mathcal{P}(x)$ via \eqref{eq:xmathSijc}, then ${0} \in \mathcal{Q}_{k}, ~ \forall k\in [r^{i}+r^{j}] $ from \eqref{eq:mathcalPij-21}. It is obvious from  $ {0} \in \mathcal{P}$ that 
\begin{align*} 
   d({0}, \partial \mathcal{P}) & < d({0}, \mathcal{P}^{*}) , 
\end{align*}
holds for any set $\mathcal{P}^{*}$ in $\mathcal{P}^{c}$. It further yields
\begin{align} 
d({0}, \partial \mathcal{P}) & <\min_{k\in [r^i+r^j]} d( {0}, \partial \mathcal{Q}_{k} / \widetilde{\mathcal{Q}}_{k} ).\label{eq:d-0-partial-wid-inq}
\end{align}
Applying the formulas \eqref{eq:d-0-partial-wid-Pij}, \eqref{eq:d-0-partial-Pij},  \eqref{eq:d-0-partial-wid-inq} in turn, we get 
\begin{align} 
    &\min_{k\in [r^i+r^j]} d({0}, \partial {\mathcal{Q}}_{k}) \nonumber   \\
    &=\min_{k\in [r^i+r^j]}  \left\{   d( {0},  \partial  \mathcal{Q}_{k} \cap \widetilde{\mathcal{Q}}_{k} ), 
    d( {0},  \partial  \mathcal{Q}_{k} / \widetilde{\mathcal{Q}}_{k})  \right\} \nonumber \\
    &=  \min \left\{ d( {0},\partial \mathcal{P}),   
    \min_{k\in [r^i+r^j]} d( {0},  \partial \mathcal{Q}_{k} / \widetilde{\mathcal{Q}}_{k} ) \right\} \nonumber \\
    &=  d( {0},\partial \mathcal{P}) .   \label{eq:minkiri}
\end{align}
Recalling the facts that whenever $ {0} \in \mathcal{P}$, it holds that  
\begin{align*} 
   \operatorname{sd}({0},  \mathcal{P} )
   &=-d({0}, \partial \mathcal{P}),  \\
   \operatorname{sd} ({0}, \mathcal{Q}_{k})
   &= -d({0},  \partial  \mathcal{Q}_{k}),  ~\forall k\in [r^i+r^j]. 
\end{align*}
Based on this and together with \eqref{eq:minkiri}, it yields 
\begin{align*} 
   {h}_{a}(x) & = \max_{k\in [r^i+r^j]}   \left\{-d({0}, \partial  \mathcal{Q}_{k})\right\} 
    = - \min_{k}  d({0}, \partial  \mathcal{Q}_{k})  \nonumber \\
  & = - d( {0},\partial \mathcal{P})  =\operatorname{sd}({0},  \mathcal{P} ) ={h}_{s}(x),    
   \end{align*}
   which means that the inequality ``$=$'' in \eqref{eq:opertsd2} has been proved. 
   
  2)  First,  by the aid of \eqref{eq:xmathSijc},  the proposition \eqref{eq:underlinesdij} can be equivalently expressed as 
   \begin{subequations} 
   \begin{align} 
   {h}_{a}(x)<0 &\Longleftrightarrow  h_{s}(x)<0, \label{eq:underhlt0} \\
   {h}_{a}(x)=0 &\Longleftrightarrow  h_{s}(x)=0, \label{eq:underheq0} \\
   {h}_{a}(x)>0 &\Longleftrightarrow  h_{s}(x)>0, \label{eq:underhgt0}
   \end{align}
   \end{subequations}
   where the first proposition naturally holds from \eqref{eq:opertsd2}, and the last two are proved as follows.
   
  Recalling \eqref{eq:mathcalPij4} in {Proposition} \ref{lem:hsxhax},  the left hand in \eqref{eq:underheq0} becomes 
    \begin{equation}\label{eq:undlinex0}
   \begin{aligned} 
   {h}_{a}(x)=0 \Longleftrightarrow  &{0} \in   \bigcap_{k\in [r^i+r^j]} \mathcal{Q}_{k}(x)=\mathcal{P}(x)    \text{ and } \\
   &\exists k\in [r^i+r^j], \text{ s.t. }
   {0} \in \partial   {\mathcal{Q}}_{k}(x),  
   \end{aligned}
   \end{equation}
   where  the right hand in \eqref{eq:undlinex0} is further equivalent to 
   \begin{align} 
    {0} 
   &\in  \mathcal{P} \cap ~\bigcup_{k } \partial  \mathcal{Q}_{k} =  \bigcup_{k} \left( \partial  \mathcal{Q}_{k} \cap \widetilde{\mathcal{Q}}_{k} \right)
   = \partial \mathcal{P}. \label{eq:mathPcap}
   \end{align}
   In fact, by recalling \eqref{eq:xmathSijc} again, Eq. \eqref{eq:mathPcap} is equivalent to the right hand in \eqref{eq:underheq0}.

   In \eqref{eq:underhgt0}, ``$\Longleftarrow$" directly comes from  \eqref{eq:mathcalPij4} in {Proposition} \ref{lem:hsxhax}, and ``$\Longrightarrow$" also holds since 
   %\begin{subequations}
   \begin{align*} 
    h_{s}(x) >0 \Longleftrightarrow&~  {0} \in \operatorname{Int} (\mathcal{P}) \\
   \Longleftrightarrow &~ {0} \in \operatorname{Int} ({\mathcal{Q}}_{k}),~\forall k\in [r^{i}+r^j]   \\
     \Longrightarrow & {h}_{a}(x)>0,   
   \end{align*}
   where the fact that $\operatorname{Int} (\mathcal{P}) = \bigcap_{k } \operatorname{Int} ({\mathcal{Q}}_{k})$ is applied. 
   \end{proof}
 
\subsection{Basic results about distance functions}\label{app:df}
\begin{lemma}\label{lem:partialK1K2}
For two nonempty closed set ${K}_{1},{K}_{2}$, %it holds that 
\begin{align*} 
  \partial ({K}_{1}\cap {K}_{2}) &=  (\partial {K}_{1} \cap {K}_{2}) \cup ( \partial {K}_{2} \cap {K}_{1}), 
\end{align*}   
and for all $x$, 
\begin{align*}
 d(x,{\partial K  }) 
  &=\min\{d(x,{\partial {K}_{1} \cap {K}_{2}}), d(x,\partial {K}_{2} \cap {K}_{1}) \}.
\end{align*}  
\end{lemma}

\begin{proof}
The relations can be proved %by referring to some basic properties in set theory, as proved in  \cite{wu2025optimization}. 
in the following procedure:
$$
  \begin{aligned}
    & \partial ({K}_{1}\cap {K}_{2}) 
   = ({K}_{1}\cap {K}_{2}) \cap  \overline{({K}_{1}\cap {K}_{2})^c} \\
  &= ({K}_{1}\cap {K}_{2}) \cap  \overline{({K}_{1}^c \cup {K}_{2}^c)}  
   = ({K}_{1}\cap {K}_{2}) \cap  ( \overline{{K}_{1}^c} \cup \overline{{K}_{2}^c}) \\
  &= ({K}_{1}\cap {K}_{2}) \cap  ( \partial {K}_{1} \cup {K}_{1}^c \cup \partial {K}_{2} \cup {K}_{2}^c)  \\
  &= \underbrace{ (  {K}_{1}\cap {K}_{2}  \cap  \partial {K}_{1})}_{= {K}_{2}  \cap  \partial {K}_{1}}   \cup
   \underbrace{ (  {K}_{1}\cap {K}_{2}  \cap {K}_{1}^c)}_{=\emptyset}  \\
  &~~      \cup  \underbrace{({K}_{1}\cap {K}_{2} \cap \partial {K}_{2})}_{= {K}_{1}  \cap  \partial {K}_{2}}  
     \cup  \underbrace{({K}_{1}\cap {K}_{2} \cap  {K}_{2}^c)}_{=\emptyset}  \\  
  &=  (\partial {K}_{1} \cap {K}_{2}) \cup (   \partial {K}_{2} \cap {K}_{1}).   
\end{aligned} $$    
Recalling the property of (unsigned) distance function $d(\cdot,\cdot)$ in \cite[Theorem 2.1 (vi)]{delfour2011shapes} that 
for all closed sets $B_{1},B_{2}$, it holds that   
$
d(x,B_{1}\cup B_{2}) =\min\{d(x, B_{1}), d(x,B_{2})  \}, ~\forall x
$. 
By choosing $\partial {K}_{1} \cap {K}_{2},  \partial {K}_{2} \cap {K}_{1}$ as $B_1, B_2$ respectively, the lemma can be proved.   
\end{proof}

\end{document}